\documentclass[11pt,reqno]{article}
\usepackage{mathtools}
\usepackage[utf8]{inputenc}
\usepackage{algorithm}
\usepackage{algorithmic}
\usepackage{booktabs} 
\usepackage{array} 
\usepackage{paralist} 
\usepackage{verbatim} 
\usepackage[table,xcdraw]{xcolor}
\usepackage{latexsym, amsmath, amssymb, amsthm, amsfonts, amscd, a4, epsfig, color}
\usepackage{blindtext}
\usepackage{graphicx}
\usepackage{subcaption}
\usepackage[utf8]{inputenc}
\usepackage[export]{adjustbox}
\usepackage{wrapfig}
\usepackage{apptools}
\usepackage{afterpage}
\usepackage{subcaption}
\usepackage{bm}
\usepackage{dsfont}

\graphicspath{}

\newtheorem{theorem}{Theorem}[section]
\newtheorem{cor}[theorem]{Corollary}
\newtheorem{lemma}[theorem]{Lemma}
\newtheorem{prop}[theorem]{Proposition}

\newtheorem{example}{Example}
\newtheorem{remark}{Remark}

\numberwithin{equation}{section}
\numberwithin{figure}{section}
\numberwithin{example}{section}
\numberwithin{remark}{section}

\newcommand{\RR}{\mathbb{R}}
\newcommand{\CC}{\mathbb{C}}
\newcommand{\NN}{\mathbb{N}}

\newcommand{\Om}{\Omega}
\newcommand{\ds}{\displaystyle}

\newcommand{\p}{\partial}
\newcommand{\pd}[2]{\frac {\p #1}{\p #2}}
\newcommand{\eqnref}[1]{(\ref {#1})}
\renewcommand{\qed}{\hfill $\Box$ \medskip}
\newcommand{\beq}{\begin{equation}}
\newcommand{\eeq}{\end{equation}}
\def\ep{\varepsilon}

\newcommand{\SingleOmega}{\mathcal{S}_{\partial\Omega}}

\newcommand{\KstarOmega}{\mathcal{K}_{\partial\Omega}^{*}}

\newcommand{\Kcal}{\mathcal{K}}
\newcommand{\Scal}{\mathcal{S}}

\newcommand{\LtwobdOmega}{L^2(\partial\Omega)}

\begin{document}

\title{Geometric series expansion of the Neumann--Poincar\'{e} operator: application to composite materials\thanks{EC acknowledges support of the U.S. National Science Foundation through
grant DMS-1715680. 
MK and ML are supported by the Basic Science Research Program through the National Research Foundation of Korea (NRF) funded by the Ministry of Education (NRF-2019R1A6A1A10073887).
}}

\author{Elena Cherkaev\thanks{Department of Mathematics, University of Utah, Salt Lake City UT 84112, USA (elena@math.utah.edu)}
\and 
Minwoo Kim\thanks{School of Electrical Engineering, Korea Advanced Institute of Science and Technology, Daejeon 34141, Republic of Korea (epsilon4b@kaist.ac.kr)}  
\and
Mikyoung Lim\thanks{Department of Mathematical Sciences,
Korea Advanced Institute of Science and Technology, Daejeon 34141, Republic of Korea (mklim@kaist.ac.kr)}}

\date{\today}
\maketitle

\begin{abstract}
The Neumann--Poincar\'{e} operator, a singular integral operator on the boundary of a domain, naturally appears when one solves a conductivity transmission problem via the boundary integral formulation. Recently, a series expression of the Neumann--Poincar\'{e} operator was developed in two dimensions based on geometric function theory \cite{Jung:2018:SSM}. In this paper, we investigate geometric properties of composite materials by using this series expansion. In particular, we obtain explicit formulas for the polarization tensor and the effective conductivity for an inclusion or a periodic array of inclusions of arbitrary shape with extremal conductivity, in terms of the associated exterior conformal mapping.  Also, we observe by numerical computations that the spectrum of the Neumann--Poincar\'{e} operator has a monotonic behavior with respect to the shape deformation of the inclusion.  
Additionally, we derive inequality relations of the coefficients of the Riemann mapping of an arbitrary  Lipschitz domain by using the properties of the polarization tensor corresponding to the domain. 
\end{abstract}

\noindent {\footnotesize {\bf AMS subject classifications.} {35J05; 30C35; 35B27; 45P05} }

\noindent {\footnotesize {\bf Key words.} 
{The Neumann--Poincar\'{e} operator; Conductivity transmission problem; Riemann mapping; Homogenization; Spectrum; Polarization tensor}}


\section{Introduction}
Various interesting wave phenomena of composite materials can be modeled as transmission problems in partial differential equations. The Neumann--Poincar\'{e} (NP) operator naturally appears when one solves a conductivity transmission problem via the boundary integral formulation. One can then obtain essential properties of composites  by analyzing the NP operator. 
For a simply connected bounded Lipschitz domain $\Omega\subset\RR^2$ and a density function $\varphi\in L^2(\partial\Omega)$, the NP operator is defined by 
\begin{equation}\label{eqn:Kstar}
	\KstarOmega[\varphi](x)=p.v.\frac{1}{2\pi}\int_{\partial\Omega}\frac{\left<x-y,\nu_x\right>}{|x-y|^2}\varphi(y)\,d\sigma(y).
\end{equation}
Here, $p.v.$ denotes the Cauchy principal value, and $\nu_x$ is the unit outward normal vector at $x\in\partial\Omega$. In higher dimensions, the NP operator is defined as a boundary integral operator in a similar manner (see, e.g, \cite{Verchota:1984:LPR}).

Recently, a series expression of the NP operator was developed in two dimensions based on geometric function theory \cite{Jung:2018:SSM}; see subsections \ref{subsec:Faber} and \ref{subsec:series} for details. This was successfully applied to the study of inclusion problems \cite{Choi:2020:ASR:preprint,Choi:2018:GME:preprint}, and the decay properties of eigenvalues of the NP operator \cite{Jung:2020:DEE}.
In the present paper, based on the geometric series expansion of the NP operator, we investigate geometric properties of composite materials by analyzing the interrelation among the interior or exterior conformal mappings, the NP operator, the polarization tensor, and the effective conductivity, where each term is associated with the geometry of an inclusion.
 
For the conductivity transmission problem with an inclusion, one can find the solution using the  ansatz
$u(x) = H(x)+\SingleOmega[\varphi](x)$, where $\SingleOmega$ indicates the single-layer potential associated with the fundamental solution to the Laplacian, and $\varphi$ involves the inversion of $\lambda I - \KstarOmega$ for some constant $\lambda$. 
This boundary integral formulation provides us the multipole expansion of the scattered field, whose coefficient of the leading term is expressed in terms of the so-called P\'{o}lya--Szeg\"{o} polarization tensor (PT), which may be defined in terms of the NP operator; we refer the reader to books \cite{Ammari:2013:MSM:book,Ammari:2004:RSI:book} and the references therein for detailed results on the PT based on the layer potential technique.

 A similar integral formula holds for a periodic interface problem, where the layer potential is now defined in terms of the periodic Green's function. 
The determination of the effective property of composite materials is one of the classical topics in 
physics \cite{Sangani:1990:CNC,Jikov:1994:HDO:book,Jeffrey:1973:CTR,Garboczi:1996:ICO}; we refer the reader to \cite{Bieberbach:1916:KDP:book} and to more references therein. 
The literature most closely related to our result is \cite{Ammari:2005:BLT} (see also \cite{Ammari:2006:ePE}), in which the effective electrical conductivity of a two-phase medium was derived using the following integral formulation:
\beq\label{effec:PT}
\sigma^*=\sigma I + fM\left(I-\frac{f}{2\sigma}M\right)^{-1}+O(f^3),
\eeq
where $f$ is the volume fraction of the inclusion, and $M$ is the PT corresponding to the inclusion and the conductivities of the matrix, $\sigma$, and the inclusion.

As one of the main results, we obtain an explicit expression of the trace of the polarization tensor in terms of the exterior conformal mapping for a simply connected, bounded Lipschitz domain with extreme (zero or infinity) or near-extreme conductivity. As an application, we find upper and lower bounds of the trace of the polarization tensor as expressed by the diameter of the domain. 
Then, using \eqnref{effec:PT}, we derive an explicit asymptotic formula for the effective conductivity of cylindrical periodic composites with inclusions with extreme conductivity, in terms only of the logarithmic capacity and the coefficient of order $-1$ of the associated exterior conformal mapping. 
We show that the asymptotic formula of the effective conductivity of a regular $n$-gon converges monotonically to that of a disk with unit area as $n$ goes to infinity.

Additionally, we derive inequality relations between coefficients of the Riemann mapping.
We remark that such inequalities are of interest in univalent function theory \cite{Duren:1983:UF, Kiryatskii:1990:SFC, Lewin:1971:BFC, Pommerenke:1975:UF:book,Thomas:2018:UF:book}, 
for example, the Bieberbach conjecture, for which we refer the reader to \cite{Bieberbach:1916:KDP:book} and to the comprehensive references therein. In the present paper, we derive relations between the coefficients of the Riemann mapping corresponding to arbitrary Jordan domains by employing the properties of the PTs corresponding to  domains with extreme conductivities.

Spectral analysis of the NP operator has drawn significant attention in relation to plasmon
resonances \cite{Ando:2016:PRF,Ammari:2013:STN,Mayergoyz:2005:EPR,Yu:2017:SDA}; the plasmon resonance happens at the spectrum of the NP operator \cite{Grieser:2014:PEP,Kang:2017:SRN}.
The NP operator $\KstarOmega$, which is symmetric on $\LtwobdOmega$ only for a disk or a ball \cite{Lim:2001:SBI}, can be symmetrized by Plemelj's symmetrization principle \cite{Khavinson:2007:PVP}.
 As a result, the spectrum of the NP operator on $H_{0}^{-1/2}(\partial\Omega)$ or $L^2_0(\p\Om)$ lies on the real axis in the complex plane and, more specifically, is contained in $(-1/2,1/2)$ \cite{Kellogg:1953:FPT,Verchota:1984:LPR} (see also \cite{Kang:2018:SPS,Krein:1998:CLO}), where $H_0^{-1/2}(\partial\Omega)$ is the Sobolev space $H^{-1/2}(\partial\Omega)$ with the mean-zero condition. 
 The spectral structure of the NP operator is different depending on the regularity of the domain. 
If $\Om$ has a $C^{1,\alpha}$ boundary, then $\KstarOmega$ is compact so that it admits the spectral decomposition 
\begin{equation}\notag
  \KstarOmega = \sum_{j=1}^\infty\lambda_j\psi_j\otimes\psi_j,
\end{equation}
where $(\lambda_j,\psi_j)$ are pairs of eigenvalues and orthonormalized eigenfunctions of $\KstarOmega$. 

If $\Om$ is merely a Lipschitz domain, then $\KstarOmega$ also admits a continuous spectrum \cite{Helsing:2017:CSN}; see also \cite{Kang:2017:SRN,Li:2018:EEN,Perfekt:2014:SBN,Perfekt:2017:ESN}. 
The spectrum of the NP operator for contacting disks was analyzed \cite{Jung:2018:SAN}.
We refer the reader to a review article \cite{Ando:2020:SAN} and the references therein for more results on the spectral analysis of the NP operator.

In this paper, we numerically compute the spectrum of the NP operator for various smooth domains using the finite section method based on the geometric series expansion of the NP operator, including the algebraic domains given by $\Psi(z)=z+\frac{a}{z^{m}}$ with  $a\in(0,\frac{1}{m})$, $m\in\NN$; see subsection \ref{subsec:examples}.
It turns out that the eigenvalues are clustered in groups whose size is $m+1$. This generalizes the results in \cite{Ammari:2019:SRN}, where the eigenvalues for the algebraic domain with small $a$ is obtained asymptotically. 
Furthermore, we observe a monotonic behavior of eigenvalues of the NP operator depending on the shape deformation of the inclusion from a disk.

The remainder of this paper is organized as follows. Section \ref{sec:preliminary} is devoted to reviewing the geometric series expansions of the layer potential operators. We then estimate the trace of the polarization tensor in section \ref{sec:polarization}, and obtain the asymptotic formula for the effective conductivity of dilute composites in section \ref{sec:Effective}. 
Inequality relations between the coefficients of the Riemann mapping is then derived in section \ref{sec:Riemann}.
By a numerical study, we investigate the shape dependence of the NP operator in section \ref{sec:NP:mono}. The paper ends with a conclusion in section \ref{ref:conclusion}.

\section{Preliminary}\label{sec:preliminary}

\subsection{Layer potential formulation}

Let $\Om$ be a simply connected bounded Lipschitz domain in two dimensions. 
Assume that $\Om$ is occupied by a homogeneous isotropic material with constant conductivity $k$ satisfying $0\leq k\neq 1 \leq \infty$. 
Consider the conductivity transmission problem
\beq\label{cond_eqn0}
\begin{cases}
\ds\nabla\cdot\left(k\chi(\Om)+\chi(\RR^2\setminus\overline{\Om})\right)\nabla u=0\quad&\mbox{in }\RR^2, \\
\ds u(x) - H(x)  =O({|x|^{-1}})\quad&\mbox{as } |x| \to \infty,
\end{cases}
\end{equation}
where $H$ is an arbitrary entire harmonic function. For the case $k=0$ or $\infty$, the zero Neumann boundary condition or the constant Dirichlet boundary condition is imposed on $\p\Om$.

One can find the solution to \eqnref{cond_eqn0} from the layer potential formulation:
\beq \label{eqn:layerpotentialansatz}
u=H+\Scal_{\p\Om}[\varphi],\eeq
where $\varphi$ is the solution to 
$$\left(\lambda I -\Kcal^*_{\p\Om}\right)[\varphi]=\pd{H}{\nu}\quad\mbox{on }\p\Om
$$
with $\lambda=\frac{k + 1}{2(k - 1 )}$. 
The NP operator, $\Kcal^*_{\p\Om}$, is given by \eqnref{eqn:Kstar}, and $\Scal_{\p\Om}[\varphi]$ denotes the single-layer potential 
$$
	\Scal_{\p\Om}[\varphi](x):=\frac{1}{2\pi}\int_{\partial\Omega}\ln|x-y|\varphi(y)\,d\sigma(y),~~~~x\in\RR^2.$$
We note that $\lambda$ is in $(-\infty,-\frac{1}{2}]\cup [\frac{1}{2},\infty)$, for which $\lambda I -\KstarOmega$ is invertible on $L^2_0(\p\Om)$ \cite{Escauriaza:1992:RTW,Verchota:1984:LPR}.

We define the P\'{o}lya--Szeg\"{o} polarization tensor (PT) in $\RR^d$ with $d=2$ as the $2\times2$ matrix $M=(M_{i,j})_{i,j=1}^2$ given by
	\begin{equation}\label{def:PT}
 M_{i,j}:= \int_{\partial \Om} y_j (\lambda I - \KstarOmega)^{-1} \left[\pd{x_i}{\nu}\right] (y) \, d\sigma(y)
	\end{equation}
	with $\lambda=\frac{k + 1}{2(k - 1 )}$. Note that the PT depends on the domain $\Om$ as well as the conductivity $k$. We write the PT as $M(\Om,\lambda)$ to highlight its dependence on the domain and the material parameter. 
The PT of $\Om$ for $d=3$, a $3\times 3$ symmetric matrix again denoted by $M$, is defined in the same way as \eqnref{def:PT}.

The solution $u$ to \eqnref{cond_eqn0} admits a far-field expansion, which one easily finds by taking the Taylor series expansion for the integral formulation \eqnref{eqn:layerpotentialansatz}: for sufficiently large $x=(x_1,x_2)$,
\begin{equation}
	u(x)-H(x)=-\frac{1} {2\pi}\nabla H(0) M(\Om,\lambda) \frac{x^T}{|x|^2} +\mbox{higher order terms}.
\end{equation}  
We refer the reader to \cite{ Ammari:2013:MSM:book,Ammari:2004:RSI:book} and references therein for further details on the layer potential formulation for the conductivity transmission problem.

\subsection{Exterior conformal mapping, Faber polynomials, and Grunsky coefficients}\label{subsec:Faber}
 We continue to assume that $\Omega$ is a simply connected bounded Lipschitz domain in $\mathbb{R}^2$.
We identify $x=(x_1,x_2)$ in $\RR^2$ with $z=x_1+i x_2$ in $\CC$. The symbols $\operatorname{Re}$ and $\operatorname{Im}$ indicate the real and imaginary parts of complex numbers, respectively.
By the Riemann mapping theorem, there exists a unique pair $(\gamma, \Psi)$ where $\gamma >0$ and $\Psi:\{w\in\mathbb{C}: |w|>\gamma\} \rightarrow \mathbb{C}\setminus\overline{\Omega}$ is a conformal map
satisfying $\Psi(\infty)=\infty$ and $\Psi'(\infty)=1$. The mapping then admits the Laurent series expansion
\beq\label{def:Psi}
\Psi(w)=w+a_0+\sum_{n=1}^\infty \frac{a_n}{w^n}
\eeq
for some complex coefficients $a_n$, which can be easily derived by reflecting the function with respect to a circle and applying the Taylor series expansion in a disk.
The quantity $\gamma$ is called the conformal radius of $\Om$ and it coincides the logarithmic capacity of $\overline{\Om}$ (see \eqnref{gamma:cap}).
From the well-known Bieberbach conjecture \cite{Bieberbach:1916:KDP:book}, it holds that
\beq\label{a_1:ineq}
|a_1|\leq\gamma^2.
\eeq
Note that simply connected bounded Lipschitz domains are Jordan domains. 
Indeed, for such a domain, there exists a constant $C>0$ such that for all $x,y \in \Omega$, there exists a smooth curve in $\Omega$ from $x$ to $y$ whose length satisfies $\leq C |x-y|$. This fact, combined with Moore's characterization of Jordan domains proves the assertion. As a result, $\Psi$ extends to the boundaries as a homeomorphism by the Caratheodory extension theorem \cite{Caratheodory:1913:GBR}.
In particular, this gives a natural parametrization of $\partial \Omega$: $\theta \mapsto \Psi(\gamma e^{i \theta})$ for $\theta \in [0, 2\pi)$.

For a complex function
 $$g(w)=w+b_0+\sum_{k=1}^\infty \frac{b_k}{w^k}$$
which is univalent on the exterior of a disk centered at the origin, 
 there exists a unique degree $m$ polynomial $F_m$ such that $F_m(g(w))$ has only one non-negative order term $w^m$ for each $m\in\NN \cup \{0\}$. In other words,
 \beq\label{Faber:condition}
 F_m(g(w)) = w^m +\mbox{negative order terms}.
 \eeq
 Indeed, $F(g(w))$ admits a Laurent series with a finite number of non-negative order terms for any complex polynomial $F$. By appropriately choosing the coefficients for $F$, we have the property $\eqnref{Faber:condition}$. 
Such $F_m$ is called the $m$-th \textit{Faber polynomial} of $g$. The concept of Faber polynomials, first introduced by G. Faber in \cite{Faber:1903:PE}, has been one of the essential elements in geometric function theory (see, e.g., \cite{Duren:1983:UF}).
The Faber polynomials can be also defined via the following generating function relation: for $z\in\Om$ and sufficiently large $w$,
\beq\notag
\frac{w g'(w)}{g(w)-z}=\sum_{m=0}^\infty \frac{F_m(z)}{w^{m}}.
\eeq

In this study, we use the Faber polynomials, denoted again by $F_m$, of the exterior conformal map $\Psi$ associated with $\Om$. We denote by $c_{m,k}$ the \textit{Grunsky coefficients} of $\Psi$ so that for each $m$, $F_m(\Psi(w))$ satisfies
\begin{equation}\label{eqn:Faberdefinition}
	F_m(\Psi(w))
	=w^m+\sum_{k=1}^{\infty}c_{m,k}{w^{-k}}.
\end{equation}
Each $F_m$ is monic and is uniquely determined by $a_0,a_1,\cdots,a_{m-1}$. For example, the first three polynomials are $F_0(z)=1,\ F_1(z)=z-a_0,\ F_2(z)=z^2-2a_0 z+(a_0^2-2a_1).$
The Grunsky identity holds for all $m,n\in\NN$:
$nc_{m,n}=mc_{n,m}.$ 
It then holds that
\beq\label{def:mu}
\mu_{m,k}=\mu_{k,m},\mbox{ where }\mu_{m,k}:= \sqrt{\frac k m} \frac {c_{m,k}}{\gamma^{m+k}}.
\eeq

\subsection{Geometric series expansion of the layer potential operators} \label{subsec:series}
In this subsection, we additionally assume that $\Om$ is enclosed by a $C^{1,\alpha}$ Jordan curve for some $0<\alpha<1$. As we assume $C^{1,\alpha}$ regularity for the boundary curve $\p\Om$, we have further regularity for $\Psi$. That is, by the Kellogg-Warschawski theorem (see, e.g., \cite{Pommerenke:1992:BBC:book}), $\Psi'$ can be continuously extended to the boundary. 
Hence, the map
\beq\label{coord:map}
(\rho, \theta)\mapsto z= \Psi\big(e^{\rho+i \theta}\big),
\eeq
from $[\ln\gamma,\infty)\times \mathbb R$ onto $\CC\setminus\Om$ is $C^1$. As $\Psi$ is a conformal (angle preserving) map, $(\rho,\theta)$ provides curvilinear orthogonal coordinates in $\CC\setminus\Om$ via \eqnref{coord:map}. 
 The scale factors with respect to $\rho$ and $\theta$ coincide with each other. We denote them by
	\begin{equation}
	h(\rho, \theta) := \left|\frac{\partial \Psi} {\partial \rho}\right| = \left|\frac{\partial \Psi} {\partial \theta}\right|.
	\end{equation}	
	
Set $\rho_0=\ln \gamma$. 
The length element on $\p\Om$ is $d\sigma(z)=h(\rho_0,\theta)d\theta$. 
For a smooth function $g(z)=(g\circ\Psi)(e^{\rho+i\theta})$, it holds that
\beq
\frac{\partial g}{\partial \nu}\Big|_{\p \Om}^+(z)=\frac{1}{h(\rho_0,\theta)}\frac{\partial }{\partial \rho}g(\Psi(e^{\rho+i\theta}))\Big|_{\rho\rightarrow\rho_0^+}.\label{eqn:normalderiv}
\eeq
We define density basis functions on $\p\Om$ in terms of the coordinate system \eqnref{coord:map}:
\begin{align*}
\zeta_0(z)&=\frac{1}{h(\rho_0,\theta)},\\
\zeta_{\pm m}(z)&= |m|^{\frac{1}{2}}\frac{e^{\pm im\theta}}{h(\rho_0,\theta)} \quad\mbox{for }m\in\NN.
\end{align*}
%
\begin{lemma}[\cite{Jung:2018:SSM}]\label{lemma:seriesexpan}
Let $\Om$ be a simply connected bounded domain having $\mathcal{C}^{1,\alpha}$ boundary with some $0<\alpha<1$. Let $z=\Psi(w)=\Psi(e^{\rho+i\theta})$ for $\rho>\rho_0=\ln \gamma$. 
We have \beq\label{Scal_zeta0}
\Scal_{\p\Om}[{\zeta}_0](z)=
\begin{cases}
\ln \gamma \quad &\mbox{if }z\in\overline{\Om},\\
\ln|w|\quad&\mbox{if }z\in\CC\setminus\overline{\Om}.
\end{cases}
\eeq 
For each $m\in\NN$, the single-layer potential associated with $\Om$ satisfies
\begin{align*}
\Scal_{\p\Om}[\zeta_m](z)&=
\begin{dcases}
-\frac{1}{2\sqrt{m}\gamma^m} F_m(z) \quad &\mbox{if }z\in \overline{\Omega},\\
-\frac{1}{2\sqrt{m}\gamma^m} \bigg( \sum_{k=1}^\infty c_{m,k} e^{-k(\rho +i\theta)} + \gamma^{2m}e^{m(-\rho+i\theta)} \bigg) \quad &\mbox{if }z\in \CC\setminus\overline{\Omega}.
\end{dcases}
\end{align*}
The series converges uniformly for all $(\rho,\theta)$ such that $\rho\geq\rho_1$ for any fixed $\rho_1>\rho_0$. Since the kernel function of $\Scal_{\p\Om}$ is real-valued, we have $\Scal_{\p\Om}[\zeta_{-m}](z)=\overline{\Scal_{\p\Om}[\zeta_{m}](z)}.$

The NP operator $\mathcal{K}^*_{\p\Om}$ admit the series expansion on $\p\Om$:
\begin{align}
\Kcal^*_{\p\Om}\left[\zeta_0\right]&=\frac{1}{2}\zeta_0, \notag\\ \label{Kstar:series}
 \KstarOmega[\zeta_m]&=\frac{1}{2}\sum_{k=1}^{\infty}{\frac{\sqrt m}{\sqrt k}}\frac{c_{k,m}}{\gamma^{m+k}}\, {{\zeta}_{-k}},\quad
 \quad \KstarOmega[\zeta_{-m}]=\frac{1}{2}\sum_{k=1}^{\infty}{\frac{\sqrt m}{\sqrt k}}\frac{\overline{c_{k,m}}}{\gamma^{m+k}}\, {{\zeta}_{k}}.
\end{align}
The series converges in the Sobolev space $H^{-1/2}(\p\Om)$ sense. 
\end{lemma}

In fact, Lemma \ref{lemma:seriesexpan} holds assuming that $\Om$ is enclosed by a piecewise $C^{1,\alpha}$ Jordan curve possibly with a finite number of corner points without inward or outward cusps. Under this boundary regularity assumption, the series in \eqnref{Kstar:series} converges in the sense of $l^2(\CC)$ space spanned by $\zeta_{\pm m}$ and $\zeta_0$.

The logarithmic capacity of a compact connected set $E$, namely $\mbox{Cap}(E)$, in the complex plane is defined as
$$-\ln\left(\mbox{Cap}(E)\right)=\lim_{z\rightarrow\infty}\left(G(z)-\ln|z|\right),$$
where $G(z)$ denotes the Green function of the Laplacian of $\overline{\CC}\setminus E$ having singularity at $z=\infty$. In other words, $G$ satisfies that $G(z)=0$ on $\p E$ and 
$$G(z)=\ln|z|-\ln\left(\mbox{Cap}(E)\right)+o(1)\quad\mbox{as }z\rightarrow\infty.$$
From \eqnref{Scal_zeta0}, $G(z)=\Scal_{\p\Om}[\zeta_0](z)-\ln \gamma$ satisfies these conditions and
\beq\label{gamma:cap}
\gamma=\mbox{Cap}(\overline{\Om}).
\eeq

From \eqnref{def:mu} and \eqnref{Kstar:series}, one can express $\mathcal{K}^*_{\p\Om}$ with respect to the basis $\left\{\cdots,\zeta_{-2},\zeta_{-1},\zeta_0,\zeta_1,\zeta_2,\cdots\right\}$ as the double-infinite self-adjoint matrix 
\begin{equation}\label{eqn:matrixKstar}
\ds\left[\Kcal^*_{\p\Om}\right]=\frac{1}{2}\begin{bmatrix}
\ds& & \vdots & & & &\vdots & &\\[1mm]
\ds& 0 & 0 & 0& 0&{\mu}_{3,1}&{\mu}_{3,2}&{\mu}_{3,3}&\\[1mm]
\ds\cdots& 0 & 0 & 0& 0&{\mu}_{2,1}&{\mu}_{2,2}&{\mu}_{2,3}&\cdots\\[1mm]
\ds& 0 & 0 & 0& 0&{\mu}_{1,1}&{\mu}_{1,2}&{\mu}_{1,3}&\\[1mm]
\ds& 0 & 0 & 0& 1&0&0&0&\\[1mm]
\ds&{\overline{{\mu}_{1,3}}}& {\overline{{\mu}_{1,2}}}\ds&{\overline{{\mu}_{1,1}}}&0&0&0&0&\\[1mm]
\ds\cdots&{\overline{{\mu}_{2,3}}}&{\overline{{\mu}_{2,2}}}\ds&{\overline{{\mu}_{2,1}}}&0&0&0&0&\cdots\\[1mm]
\ds&\overline{{\mu}_{3,3}}& {\overline{{\mu}_{3,2}}}\ds&{\overline{{\mu}_{3,1}}}&0&0&0&0&\\[1mm]
\ds& &\vdots & & & &\vdots& &
\end{bmatrix}.
\end{equation}

\section{Polarization tensor for a planar inclusion with extreme or near-extreme conductivity}\label{sec:polarization}
In this section we express the polarization tensor of a simply connected bounded Lipschitz domain with extreme or near-extreme conductivity in terms of the external conformal map and derive properties of the trace of the polarization tensor. 
\subsection{Explicit formula of the polarization tensor}

As is well-known as \textit{the Hashin-Shtrikman bound}, the PT satisfies (see \cite{Capdeboscq:2004:RSR,Lipton:1993:IEE})
		\beq
		\label{eq:upperhs}
		\ds \frac{1}{k-1} \operatorname{tr} (M) \leq \left(d-1 + \frac 1 k\right) |\Omega|
		\eeq
		and
		\beq
		\label{eq:lowerhs}
		\ds (k-1)\operatorname{tr} (M^ {-1}) \leq  \frac{d-1+k} {|\Omega|},
		\eeq
where $|\Om|$ denotes the volume of $\Om$ and $\mbox{tr}$ means the trace of a matrix. Attainability by simply connected domain of the eigenvalues of the PT satisfying \eqnref{eq:upperhs}--\eqnref{eq:lowerhs} was numerically verified for $d=2$ \cite{Ammari:2006:ASC}.
The lower Hashin--Shtrikman bound is attained (i.e., the equality holds in \eqnref{eq:lowerhs}) if and only if $\Om$ is an ellipse, asserted as the Polya--Szeg\"o conjecture \cite{Polya:1951:IIM:book}; this conjecture was proved by Kang and Milton \cite{Kang:2008:SPS}.

As one of our main objectives of this paper, we investigate the property of $\operatorname{tr}(M)$ and $\operatorname{tr}(M^{-1})$ for a planar inclusion with extreme conductivity. 
We note that for $k\rightarrow 0$ or $\infty$, \eqnref{eq:lowerhs} becomes 
\beq\label{2nd:extreme}
		\left|\operatorname{tr} (M^ {-1})\right| \leq \frac {1} {|\Omega|}.
		\eeq
However, the upper Hashin-Shtrikman bound, \eqnref{eq:upperhs}, becomes vacuous as $k \to 0$ or $\infty$. In the present paper, instead of using \eqnref{eq:upperhs}, we characterize the trace of the PT for an inclusion with extreme or near-extreme conductivity by using the following explicit expression of the PT:

\begin{theorem}\label{theorem:DJM1}
Let $\Omega$ be a simply connected, planar bounded Lipschitz domain with $k=0,\infty$ (i.e., $\lambda=\pm\frac{1}{2}$).  We denote the exterior conformal mapping associated with $\Om$ as in \eqnref{def:Psi}.
Then, the polarization tensor of $\Om$ is
 \begin{equation} \label{eq:explicitformula}
M(\Om,\pm\frac{1}{2}) = 2\pi \begin{bmatrix}
\pm \gamma^2 + \operatorname{Re}(a_1) & \operatorname{Im}(a_1) \\[1.5mm]
\operatorname{Im}(a_1) & \pm \gamma^2 - \operatorname{Re}(a_1)
\end{bmatrix}.
\end{equation}
\end{theorem}
\begin{proof}
For domains with $C^{1,\alpha}$ boundary curves, \eqnref{eq:explicitformula} was derived in \cite{Choi:2018:GME:preprint} by using the series expansions of the layer potential operators described in subsection \ref{subsec:series} (see also \cite{Choi:2020:ASR:preprint} for the relations between higher order terms). 
 For general Lipschitz domains, this is a direct consequence of the exact relations between the coefficients of the conformal mapping and the generalized polarization tensors obtained in \cite[Proposition 2.3]{Choi:2018:CEP} or \cite{Kang:2015:CCM}.
 Although, it was dealt with only in the case $k=0$ in \cite{Choi:2018:CEP,Kang:2015:CCM}, one can extend the result to the case $k=\infty$ with minor changes in the proof.
\end{proof}

As a corollary of Theorem \ref{theorem:DJM1}, we find that
\beq\label{eqn:tr:inv}
\operatorname{tr} (M^{-1}) = \pm \frac 1 \pi \frac{\gamma^2 } {\gamma^4 - |a_1|^2}.
\eeq
As the volume of $\Om$ satisfies
\beq\label{Om:vol}
|\Om|=\pi \Big(\gamma ^2 - \sum_{n \geq 1} n \frac{|a_n|^2}{\gamma^{2n}}\Big),
\eeq
equality holds in \eqnref{2nd:extreme} if and only if $a_n=0$ for all $n\geq 2$ (see \cite{Choi:2018:GME:preprint}). This is an alternative proof of the P\'{o}lya--Szeg\"o conjecture for a planar inclusion with extreme conductivity. Also note that the eigenvalues of $M(\Om,\pm\frac{1}{2})$, namely $\tau_1$ and $\tau_2$, are
\beq\label{tau1tau2}
\tau_1=\pm 2\pi\left(\gamma^2+|a_1|\right),\quad \tau_2=\pm 2\pi\left(\gamma^2-|a_1|\right).
\eeq

Let us now analyze the trace of the polarization tensor for $k\rightarrow 0$ or $\infty$. 
As a direct consequence of Theorem \ref{theorem:DJM1}, it also holds for $k=0,\infty$ (i.e., $\lambda=\pm \frac{1}{2}$) that
$
\operatorname{tr}(M) = \pm 4 \pi \gamma^2.
$
Furthermore, we validate that this relation approximately holds also for the near-extreme conductivity
case:
\begin{theorem} \label{cor:obs}
Let $\Omega$ be a simply connected bounded Lipschitz domain  with the constant conductivity $k$. Let $\gamma$ be the logarithmic capacity of $\p\Om$. 
For $\lambda\approx\pm\frac{1}{2}$, it holds that
	\begin{equation}\label{eqn:tr:approx}
	\operatorname{tr}(M)=\pm4\pi\gamma^2 + O\Big(|\lambda|-\frac{1}{2}\Big).
	\end{equation}
\end{theorem}
\begin{proof}
Let $|\lambda_k|\geq\frac{1}{2}$ ($k=1,2$) and define
 $$\varphi_{\lambda_k}:=(\lambda_k I - \KstarOmega)^{-1} \Big[\pd{x_i}{\nu}\Big].$$
 Then, it holds from the spectral resolution of the NP operator (see the proof of Theorem 3.3 in \cite{Kang:2016:SPN}) that
$$\left\|\varphi_{\lambda_1}-\varphi_{\lambda_2}\right\|_{H^{-1/2}(\p\Om)}\leq C\left|\lambda_1-\lambda_2\right|$$
for some constant independent of $\lambda_1$, $\lambda_2$. 
From the definition \eqnref{def:PT}, this proves the theorem. 
\end{proof}

\subsection{Properties of the trace of the polarization tensor}

One can estimate $\gamma$ in terms of $\mbox{diam}(\Om)$ thanks to previous literature (e.g., \cite{Schaeffer:1950:CRS:book}). 
We give the proof at the end of this subsection for the reader's convenience.
\begin{lemma} \label{lemma:diambound}
The conformal radius $\gamma$ satisfies that
	 $\frac{1}{4}\operatorname{diam}(\Om)\leq\gamma\leq\operatorname{diam}(\Om).$
\end{lemma}

From Theorem \ref{cor:obs} and Lemma \ref{lemma:diambound}, we conclude that a large value of $\operatornamewithlimits{diam}(\Omega)$ corresponds to a large value of $|\operatorname{tr}(M)|$, as stated in the theorem below. 
\begin{theorem}\label{theorem:diambound}
For $\lambda=\pm\frac{1}{2}$, we have $$\frac{\pi}{4} (\operatorname{diam}(\Omega))^2\leq \left|\operatorname{tr} ( M)\right| \leq 4 \pi (\operatorname {diam} (\Omega))^2.$$
\end{theorem}

Theorem \ref{theorem:diambound} encapsulates the intuition that domains which are close to the upper Hashin--Shtrikman bound are ``thin" and vice versa. The following proposition reinforces this idea by illustrating that ``fat'' domains are close to the lower Hashin--Shtrikman bound for a planar domain with extreme conductivity, \eqnref{2nd:extreme}.
\begin{prop} [Shape change toward a disk]\label{prop:path}
Let $\Omega$ be a simply connected bounded Lipschitz domain with $k=0,\infty$. We denote by $\Om_r$ the family of domains enclosed by the curve $\{z=\Psi(w): |w|=r\}$, $r\geq \gamma$.

Then, as the shape of the domain changes toward a disk, the trace of the PT corresponding to $\Om_r$ converges to the lower Hashin--Shtrikman bound: 
 \beq\label{eqn:Om_r}
 \ds|\Omega_r|\big| \operatorname{tr}\left(M^{-1}(\Om_r)\right)\big| \to 1\quad\mbox{as }r\rightarrow\infty.
 \eeq
On the other hand, the difference between the two eigenvalues of $M(\Om_r)$ is constant independent of $r$:
\beq\label{tau:diff}
\big|\tau_1(\Om_r) - \tau_2(\Om_r)\big| =  4\pi|a_1|\quad\mbox{for all }r\geq\gamma.
\eeq
\end{prop}
\begin{proof}
From \eqnref{eqn:tr:inv}, \eqnref{Om:vol} and \eqnref{tau1tau2}, one can easily find 
\eqnref{eqn:Om_r} and \eqnref{tau:diff}.
\end{proof}

While the family of domains $\Om_r$ has the polarization tensor gradually changes as $r\rightarrow\infty$, 
one can modify the shape of any domain to have a large diameter by stretching only a small portion of the domain; see, e.g., Figure \ref{fig:insta}. Hence, we observe the following as a direct consequence of Theorems \ref{cor:obs} and \ref{theorem:diambound}:
\begin{cor} [Instability with respect to the shape perturbation]
Let $\lambda \approx \pm \frac{1}{2}$.
For any Lipschitz domain $\Om$ and $N\in\NN$, there exists a shape perturbation of $\Om$, namely $\widetilde{\Om}$, such that $$\left|\Om\triangle\widetilde{\Omega}\right|<\frac{1}{N}\quad\mbox{and}\quad \left|\operatorname{tr}(M(\Om,\lambda))-\operatorname{tr}(M(\widetilde{\Om},\lambda))\right|>N.$$
\end{cor}
\begin{figure}[h!]
	\hskip 3cm
	\includegraphics[scale=0.36]{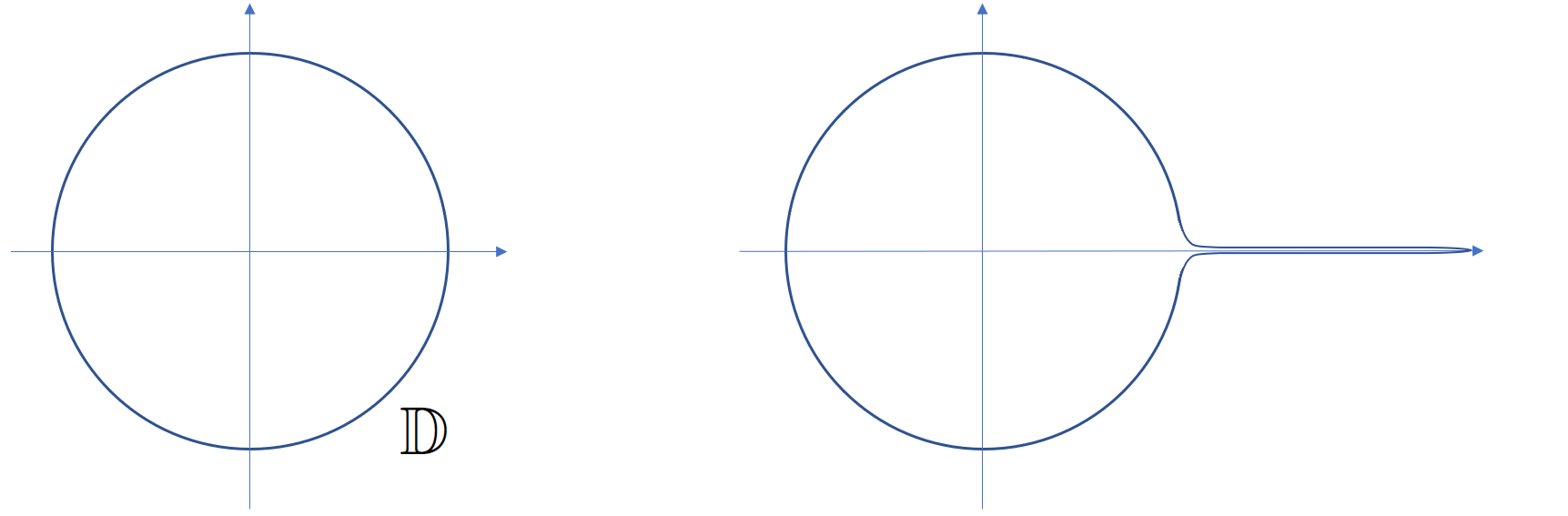}
	\caption{A disk (left figure) and its shape perturbation (right figure). While the perturbation is small in area sense, there is a big difference in the diameters of the two domains.}
	\label{fig:insta}
\end{figure}

\noindent{\textbf{Proof of Lemma \ref{lemma:diambound}.} 
Set $R:=\mbox{diam}(\Om)=\sup_{x,y\in\Om}|x-y|$. Fix a point $z_0\in\Om$. 
The map $f:\mathbb{D} \rightarrow \CC$ given by 
\begin{align*}
f(z):=\frac{1}{\Psi(\frac{\gamma}{z})-z_0}
=\frac{z}{\gamma}\Big(1-\frac{a_0-z_0}{\gamma}z+\mbox{higher order terms}\Big)
\end{align*}
with $f(0)=0$ is a conformal map from the unit disk onto the range of $f$ satisfying $f'(0)=\frac{1}{\gamma}$. 
From the Bieberbach's conjecture for the second coefficient, we have $|a_0-z_0| \leq 2\gamma$. Since $z_0 \in \Omega$ is arbitrary, we conclude that $R \leq 4\gamma$.
\smallskip

Note that $\big\{\Psi\big(\frac{\gamma}{z}\big)-z_0:|z|<1\big\}=\CC\setminus(\overline{\Om}-z_0)\supset\{z:|z|>R\}$. 
We then observe that the map
$ z \mapsto f^{-1}\big(\frac{z}{R}\big)$ is a complex analytic function from the unit disk to the unit disk, and sends $0$ to $0$. From the Schwarz lemma, we have $|f'(0)R|^{-1}\leq 1$, i.e., $\gamma \leq R$. 
\qed

\section{Effective conductivity of dilute composites}\label{sec:Effective}

In this section we determine the effective or macroscopic property of a two-phase medium by using Theorem \ref{theorem:DJM1}.

\subsection{Asymptotic formula for the effective conductivity}
Let $Y=]-\frac{1}{2},\frac{1}{2}[^2$ denote the unit cell in $\RR^2$ and $D=\rho B$, $0<\rho<1$, be a subset of $Y$, where $B$ is a reference bounded Lipschitz domain containing the origin with $|B|=1$. Let the conductivity $\sigma$ be periodic with the periodic cell $Y$ and, on $Y$,
$$\sigma=1+(k-1)\chi(D),$$
i.e., $Y\setminus\overline{D}$ and $D$ have the constant conductivities $1$ and $k$ ($0<k\neq1<+\infty$), respectively. The effective conductivity, namely $\sigma^*=({\sigma}^*_{ij})_{i,j=1,2}$, of the periodic conductivity $\sigma\big(\frac{x}{\ep}\big)$ as $\ep\rightarrow 0$ is then defined to be (see, e.g., \cite{Jikov:1994:HDO:book})
\beq\label{effec:integral}
\sigma^*_{ij}=\int_{Y}\sigma\nabla u_i\cdot\nabla u_j\,dx,
\eeq
where $u_i$ ($i=1,2$) is the unique solution to 
\beq\label{eq:effectiveconductivitypde}
\begin{cases}
\ds\nabla\cdot\sigma\nabla u_i=0\quad\mbox{in }Y,\\
\ds u_i-y_i\mbox{ is periodic},\\
\ds \int_Y u_i\, dx=0.
\end{cases}
\eeq

When $B$ is a ball, the corresponding effective conductivity is given by the Maxwell--Garnett formula \cite{Sangani:1990:CNC}
$$\sigma^*=1+\rho^2\frac{2(k-1)}{k+1}+2\rho^4\frac{(k-1)^2}{(k+1)^2}+o(\rho^4).$$
This formula has been extended to inclusions of general shape with Lipschitz boundaries \cite[Theorem 5.1]{Ammari:2005:BLT}:
\beq\label{formula:effective:ammari}
 {\sigma}^*=\left(I+\rho^{2}M+\frac{\rho^{4}}{2}M^2\right)+O(\rho^{6}),
\eeq
where $M$ is the P\'{o}lya--Szeg\"{o} polarization tensor corresponding to $B$ and conductivity $k$ with the background conductivity $1$.

By applying Theorem \ref{theorem:DJM1} to \eqnref{formula:effective:ammari}, we have the following. 
\begin{theorem}\label{thm:effective}
Let $D=\rho \Om$, $\rho^2{|\Om|}<1$, with the conductivity $k=0$ or $\infty$ (i.e., $\lambda=\pm\frac{1}{2}$). Then, $\sigma^*=\sigma^*(D,\lambda)$ admits the asymptotic expansion
\begin{align}\notag
\sigma^*\left(D,\pm 1/2\right)=I+ 2\pi\rho^2&\begin{bmatrix}
\ds\pm \gamma^2 + \operatorname{Re}( a_1) & \ds\operatorname{Im}( a_1) \\[1.5mm]
\ds\operatorname{Im}( a_1) & \ds \pm \gamma^2 - \operatorname{Re}( a_1)
\end{bmatrix}\\[2mm]
+2\pi^2\rho^4&\begin{bmatrix}
\ds\gamma^4\pm2\gamma^2\operatorname{Re}(a_1)+| a_1|^2 &\ds \pm2\gamma^2\operatorname{Im}( a_1)\\[1.5mm]
\ds\pm 2\gamma^2\operatorname{Im}( a_1)  & \ds\gamma^4\mp 2\gamma^2\operatorname{Re}( a_1)+| a_1|^2
\end{bmatrix} + O(\rho^6|\Om|^3).\label{eq:effectiveconductivity}
\end{align}
Recall that $|\Om|$ is given by $
|\Om|=\pi \left(\gamma ^2 - \sum_{n \geq 1} n \frac{|a_n|^2}{\gamma^{2n}}\right)$. 
\end{theorem}
\begin{proof}
The scaled domain $\widetilde{\Om}:=\frac{\Om}{|\Om|^{\frac 1 2}}$ (hence, $|\widetilde{\Om}|=1$) has the exterior conformal mapping 
$$\widetilde{\Psi}(w)=w+\frac{a_0}{|\Om|^{\frac 1 2}}+\sum_{n=1}^\infty\frac{a_n}{|\Om|^{\frac {n+1} 2 }}\frac{1}{w^n}\quad\mbox{for }|w|>\widetilde{\gamma}:=\frac{\gamma}{|\Om|^{\frac 1 2}}.$$ 
The polarization tensor of $\widetilde{\Om}$ is then given by
 \begin{equation} 
\widetilde{M}:=M(\widetilde{\Om},\pm\frac{1}{2}) = \frac{2\pi}{|\Om|} \begin{bmatrix}
\pm \gamma^2 + \text{Re}( a_1) & \text{Im}( a_1) \\[1.5mm]
\text{Im}( a_1) & \pm \gamma^2 - \text{Re}( a_1)
\end{bmatrix}.
\end{equation}
From \eqnref{formula:effective:ammari} with $\rho$ replaced by $|\Om|^{\frac{1}{2}}\rho$, 
this completes the proof. 
\end{proof}

Let us denote the lower order terms of $\sigma^*-I$ as $A=A(\rho, \gamma, a_1,\lambda)$ so that 
\beq\label{def:A}
\sigma^*=I+A+O(\rho^6|\Om|^3).
\eeq
The matrix $A$ has the trace and the determinant as follows:
\begin{align*}
	 \operatorname{tr}(A) &= \pm 4\pi \gamma^2 \rho^2  + 4 \pi^2 \left(\gamma^4 +|a_1|^2\right) \rho^4 ,\\
\det(A) &= 4 \pi^2 \rho^4 \left(\gamma ^4 -|a_1|^2\right)   \pm 8 \pi^3 \gamma^2 \rho^6 \left( \gamma^4 - |a_1|^2 \right)  + 4 \pi^4 \rho^8 \left( \gamma^4 - |a_1|^2 \right)^2.
	\end{align*}
In view of \eqnref{a_1:ineq}, $A$ is invertible if $|a_1|<\gamma^2$ and $\rho$ is sufficiently small. 
Assuming further that $k=\infty$, we have
\begin{align*}
\operatorname{tr}(A^{-1})
&= \frac{1+2\pi\gamma^2 - \rho^2  X}{\rho^2 X+2\pi\gamma^2\rho^4X+\pi \rho^6 \gamma^2 X^2},\ X=\pi\left(\gamma^2 - \frac{|a_1|^2}{\gamma^2}\right)\\
&\leq \frac{1+2\pi\gamma^2 - \rho^2  |\Om|}{\rho^2 |\Om|+2\pi\gamma^2\rho^4|\Om|+\pi \rho^6 \gamma^2 |\Om|^2},
\end{align*}
where the equality holds if and only if $a_n =0$ for all $n \geq 2$. In other words, assuming that $\gamma$, $\rho$ and $|\Om|$ are fixed, $\operatorname{tr}(A^{-1})$ attains the maximum when $D$ is an ellipse. 
This can be considered an extension of the Polya-Szeg\"{o} conjecture to the higher-order expansion of the effective conductivity.

\subsection{Examples}
We provide the asymptotic expansions of the effective conductivity for some simple polygons, whose exterior conformal mapping can be easily expressed as the Schwarz--Christoffel integral. 
\subsubsection{A rectangle}
Consider a rectangle $\Om$ with width $\sqrt{3}$ and height $\frac{1}{\sqrt{3}}$ centered at the origin, where the vertices are given by $\frac{\sqrt 3}  2 \pm\frac 1 {2 \sqrt 3}i$ and $-\frac{\sqrt 3}  2 \pm \frac 1 {2 \sqrt 3}i$.
By numerically approximating the corresponding Schwarz--Christoffel mapping using the MATLAB SC Toolbox, we have
$a_1 \simeq  0.20439$ and $\gamma \simeq 0.66273.$
This leads to the asymptotic expansion of the effective conductivity for $D=\rho\Om$:
\begin{align*}
&\sigma^*\left(D,1/2\right)= I + 
\begin{bmatrix}
 4.0438 & 0 \\
  0 &  1.4754\\
\end{bmatrix} \rho^2 +\begin{bmatrix}
 8.1763  &                 0\\
0   &1.0885
\end{bmatrix} \rho^4 + O(\rho^6),\\
&\sigma^*\left(D,- 1/2\right) = I + 
\begin{bmatrix}
-1.4754    &               0\\
0 & -4.0438
\end{bmatrix} \rho^2 +\begin{bmatrix}
  1.0885     &             0\\
0  & 8.1763
\end{bmatrix} \rho^4 + O(\rho^6).
\end{align*}

In Figure \ref{rectangle:n-gon}, we illustrate the results on a rectangle and a regular $4$-gon. We compare the leading term in the asymptotic expansion in \eqnref{eq:effectiveconductivity} with the effective conductivity obtained by evaluating \eqnref{eq:effectiveconductivitypde}, for which we numerically solve the PDE problem \eqref{eq:effectiveconductivitypde} by a finite difference method (FDM). The conductivity is set to be $k=0$.  The left figure illustrates that the $\rho^4$ term, although negligible for $\rho \simeq 0$, is indeed significant for mid-range $\rho$ values.

\begin{figure}[h!]  \centering

	\begin{subfigure}{0.505\linewidth}	
		\hskip -.95cm
		\includegraphics[scale=0.187]{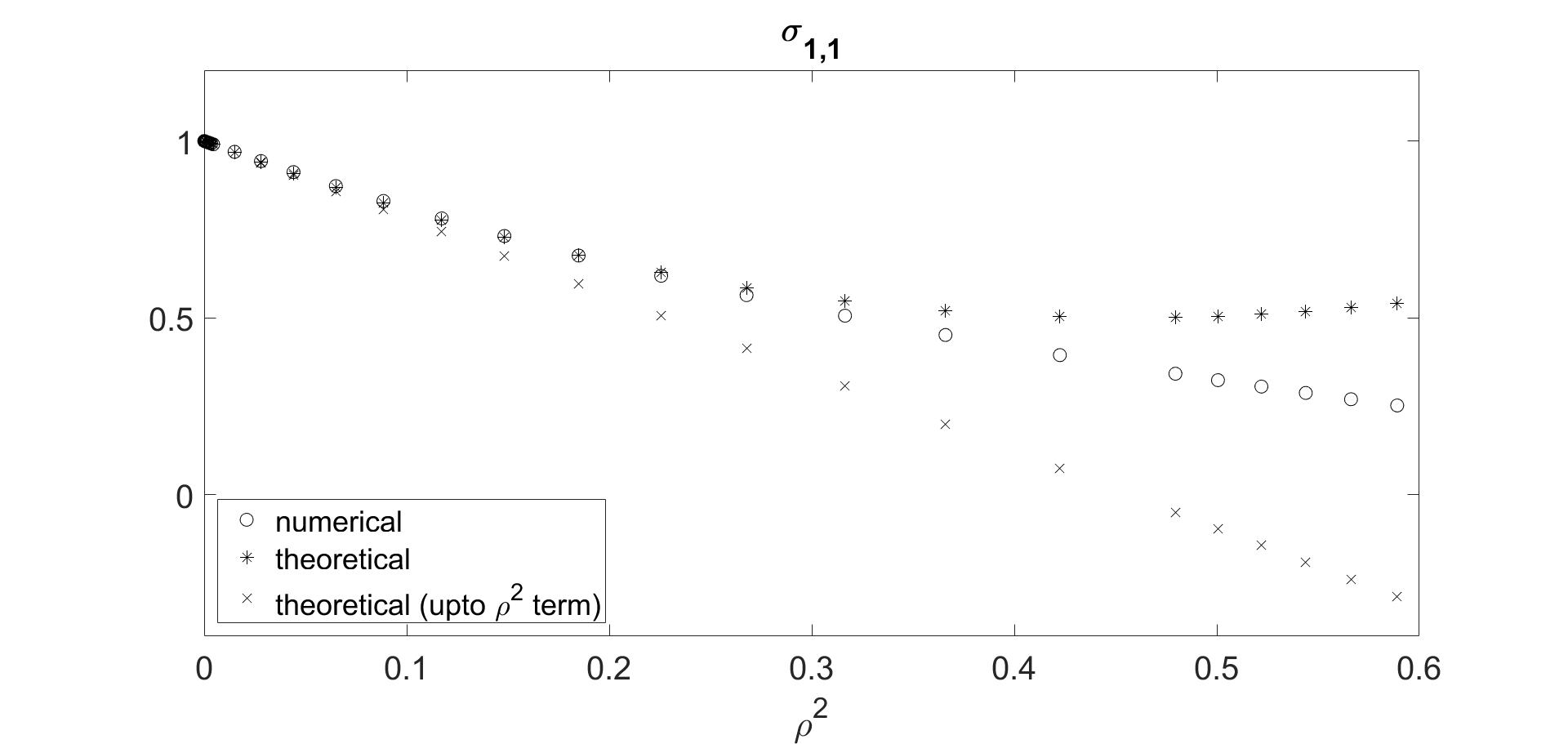}
	\end{subfigure}%
	\begin{subfigure}{\linewidth}	
	\hskip -.07cm
		\includegraphics[trim={4cm 0 0 0},clip, scale=0.187]{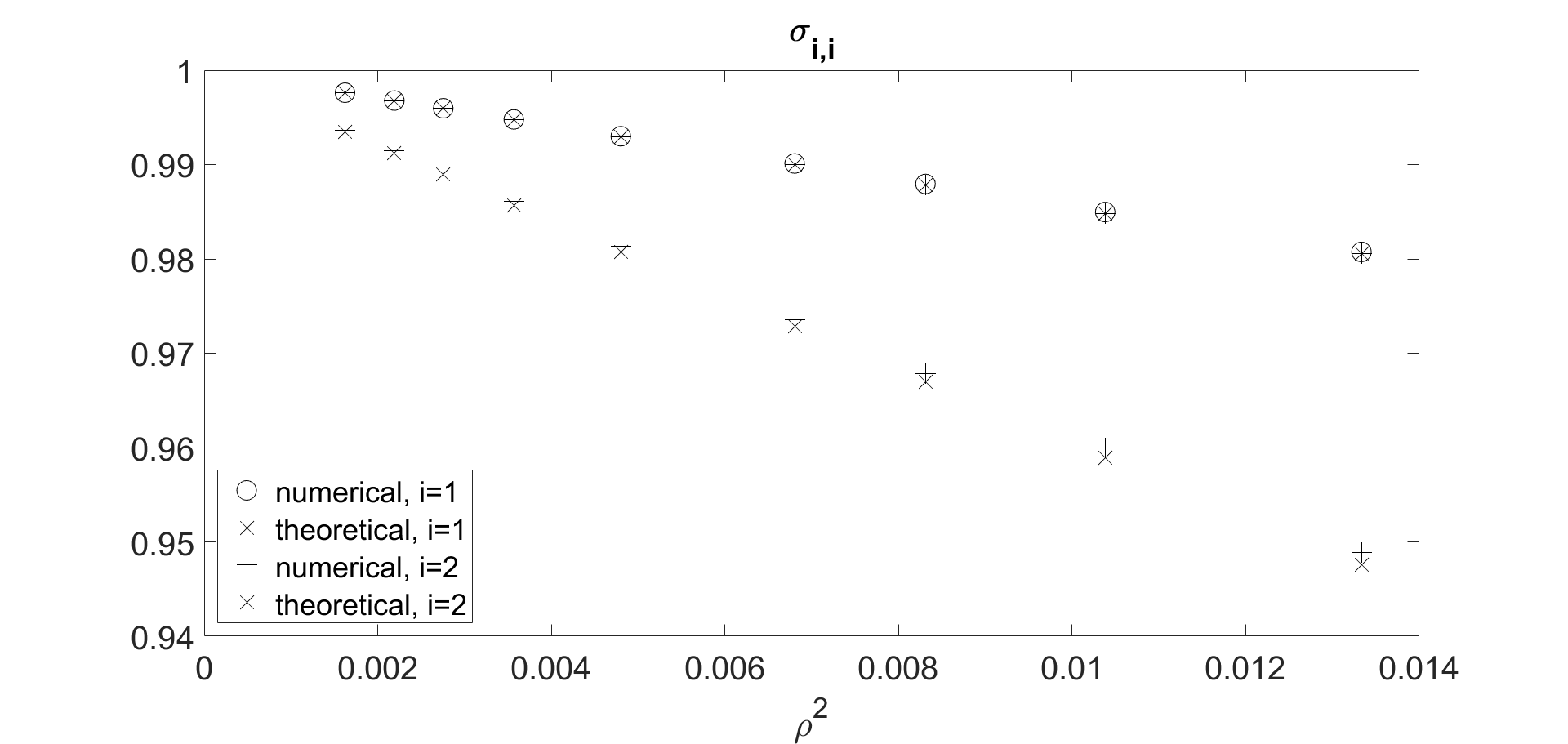}
	\end{subfigure}
	\caption{Numerical approximations of $\sigma_{i,i}$, $i=1,2$ for a rectangle width $\sqrt{3}$ and height $\frac{1}{\sqrt{3}}$ (left figure) and a regular $4$-gon (right figure). The edges are set parallel to the axes. In the left figure, $\sigma_{1,1}=\sigma_{2,2}$ by symmetry.
	}\label{rectangle:n-gon}
\end{figure}

}

\subsubsection{Regular $n$-gons}

\begin{lemma}\label{lemma:symmetric}
For $\Om$ with rotation symmetry of angle $\theta$ satsifying $e^{2i\theta}\neq1$, we have $a_1=0$ so that $I+A$ is diagonal (see \eqnref{def:A}). 
\end{lemma}
\begin{proof} 
	We may assume the rotational center of the domain is at the origin. For such a domain $\Om$ and its associated exterior conformal mapping $\Psi[\Om]$, $e^{i \theta} \Psi[\Omega] \big(e^{-i \theta} z\big) $ is also an exterior conformal map. In particular, this map equals $\Psi[\Omega](z)$ and it holds that $a_1 e^{2i \theta } = a_1$. This implies $a_1=0$.
\end{proof}Lemma \ref{lemma:symmetric} holds in particular for domains with rotational symmetry of order greater than $2$. Let $\Om$ be a regular $n$-gon domain with side length $L$, then $a_1=0$ from Lemma \ref{lemma:symmetric}. From \cite{Polya:1951:IIM:book} and \cite[Corollary 11.1]{Pommerenke:1975:UF:book}, we have
\beq\label{gamma:n}
\gamma = \frac 1 {4 \pi} \frac{\Gamma ^2\left(\frac 1 n\right)}	{\Gamma \left(\frac 2 n\right)} L, \quad\mbox{where }L=\sqrt{\frac{4 \tan (\frac \pi n)}{n}} \mbox{ assuming }|\Om|=1.
\eeq
 
 \begin{lemma}\label{n_gon:gamma:deri}
The conformal radius $\gamma(n)$ of a regular $n$-gon with unit area, given by \eqnref{gamma:n}, strictly decreases to $\frac{1}{\sqrt{\pi}}$ as $n$ increases to $\infty$. Note that $\frac{1}{\sqrt{\pi}}$ is the conformal radius of a disk with unit area. \end{lemma}
\begin{theorem}
Let $\Om$ be a regular $n$-gon of unit volume centered at the origin with the conductivity $k=0,\infty$ ($n\geq3$). Then, the corresponding effective conductivity for $D=\rho \Om$ is 
\beq\label{eq:eff_cond_ngon}
\sigma^*
 = \left(1 \pm      \frac{1}{2\pi}\frac{\Gamma^4\left(\frac 1 n\right)}{\Gamma^2\left(\frac 2 n\right)}\frac{\tan\left(\frac \pi n\right)}{n}\, \rho ^2 
+\frac{1}{8\pi^2}\frac{\Gamma^8\left(\frac 1 n\right)}{\Gamma^4\left(\frac 2 n\right)} \frac{\tan^2 \left(\frac \pi n\right)}{ n^2} \, \rho ^4\right)I + O(\rho^6),
\eeq
where the symbols $+$ and $-$ correspond to $k=\infty$ and $k=0$, respectively. 
The asymptotic of $\sigma^*$ up to $\rho^4$-term converges to that corresponding to the unit disk as $n\rightarrow\infty$ in a strictly decreasing manner for $k=\infty$, and a strictly increasing manner for $k=0$ with sufficiently small $\rho$. 
\end{theorem}
\begin{proof}
From Theorem \ref{thm:effective} and \eqnref{gamma:n},  we observe \eqnref{eq:eff_cond_ngon}. Monotonicity properties follow from Lemma \ref{n_gon:gamma:deri}. 
\end{proof}
For example, for the cases $n=3$ and $n=4$ we respectively have
\begin{align*}
\sigma^* &= \Big(1 \pm  \frac{3 \sqrt 3\, \Gamma^6(1/3)} {8 \pi^{3}} \rho^2 + \
\frac{27\, \Gamma^{12}(1/ 3)}{128 \pi^6} \rho^4\Big) I +O(\rho^6),\\
\sigma^*&=\Big(1 \pm   \frac{\Gamma^4(1/4)}{8 \pi^2}\rho^2 +  \frac{\Gamma^8(1/4)}{128 \pi^4} \rho^4\Big)I + O(\rho^6).
\end{align*}

\noindent{\textbf{Proof of Lemma \ref{n_gon:gamma:deri}.} 
The digamma function admits the series expansion
$	\frac{\Gamma'(z)}{ \Gamma(z)}=-C-\frac 1 z + \sum_{n = 1}^\infty \left(\frac 1 n - \frac 1 {n+z}\right)$, $z \neq 0, -1,-2, \dots,$
with the Euler--Mascheroni constant $C$  (see \cite[formulas 6.3.1, 5, and 16]{Abramowitz:1964:HMF:book}). 
In particular, it holds for $x>2$ that
\begin{align}\notag
\gamma'(x)
	&= \frac{1}{2\pi} 
	 \frac {\Gamma^2(\frac 1 x)} {\Gamma(\frac 2 x) }\frac{\sqrt{\tan (\frac \pi x) }}{ x^{\frac 3 2}}\, \left( \frac 1 2 + \sum_{n=1}^\infty \frac{2}{(nx+1)(nx+2)}  - \frac \pi {x \sin ( \frac{2\pi} {x} )} \right),
	 \end{align}
where $\gamma(x)$ is the function given by \eqnref{gamma:n} with $n$ replaced by a real variable $x$. 
 We claim 
	 \beq\label{gamma:ineq}
	 \gamma'(x)<0\quad\mbox{for }x>2.
	  \eeq	 
Because of $ \sum_{n=1}^\infty \frac{2}{(nx+1)(nx+2)}\leq \sum_{n=1}^\infty \frac{2}{n^2 x^2}=\frac{\pi^2}{3}\frac{1}{x^2}$, we only need to show $\frac 1 2 + \frac {\pi^2} {3x^2} - \frac{\pi}{x \sin(\frac {2\pi} x)} <0$.
Note that
\begin{align*}
&f(s):=(\pi^2 s^2 +1) \sin(2\pi s) - 2\pi s \cos(2 \pi s) >0,\\
&\frac{d}{ds}\left[{6s^2}\,\frac{d}{ds}\left(\Big(\frac{1}{2s}+\frac{\pi^2 s}{3}\Big)\sin (2\pi s)\right)\right]= - 8\pi^2 s f(s)<0 \quad\mbox{for }s\in(0,\frac{1}{2}).
\end{align*}
Using this, one can easily find that $(\frac 1 {2s} + \frac {\pi^2 s}{ 3})  \sin(2\pi s) < \pi  $ for  $s\in (0, \frac 1 2)$, which is equivalent to $\frac 1 2 + \frac {\pi^2} {3x^2} - \frac{\pi}{x \sin(\frac {2\pi} x)} <0$ for $x>2$. This proves \eqnref{gamma:ineq}. 
Hence, we conclude that $\gamma(n)$ is strictly decreasing w.r.t. $n\geq 3$. Furthermore, as $\Gamma(s)=\frac{1}{s}+\mbox{h.o.t.}$ near $s=0$, $\gamma(n)$ converges to ${\frac{1}{\sqrt{\pi}}}$ as $n\rightarrow\infty$.
\qed

\section{Inequality relations between the coefficients of the Riemann mapping}\label{sec:Riemann}
The following relations on the polarization tensor are well-known:
\begin{lemma}[\cite{Ammari:2007:PMT:book}] \label{lemma:tensor_bounds}
	Suppose that $B$ is a simply connected bounded Lipschitz domain and $0 \leq k\neq 1\leq \infty$. Let $\lambda=\frac{k + 1}{2(k - 1 )}$. Assume that $I$ is a finite set of multi-indices and $\{a_\alpha\} _{\alpha \in I}$ are constants such that the function $H(x) = \sum_{\alpha \in I}  a_\alpha x^\alpha$ is harmonic. Then, the following inequality holds:
	\beq\notag
	\left(1-\frac{1}{k}\right) \int_B | \nabla H|^2 \leq \sum _{i, j \in I}a_i a_j M_{ij}(B,\lambda) \leq (k-1) \int_B | \nabla H|^2 .
	\eeq
\end{lemma}
\begin{lemma}[\cite{Ammari:2005:PTT}] \label{lemma:monotonicity}
		Assume the same conditions for $k$, $\lambda$, $I$, and $\{a_\alpha\} _{\alpha \in I}$ as in Lemma \ref{lemma:tensor_bounds}. Suppose that $B$ and $B'$ are two simply connected bounded Lipschitz domains satisfying $B' \subseteq B$. Then, the following inequalities hold:
\begin{equation}
 	\begin{aligned}
 	\sum_{i, j \in I} a_i a_j M_{ij} (B,\lambda) \geq  \sum_{i, j \in I} a_i a_j M_{ij} (B',\lambda) & \text{ if $k>1$}, \\
 	\sum_{i, j \in I} a_i a_j M_{ij} (B,\lambda) \leq  \sum_{i, j \in I} a_i a_j M_{ij} (B',\lambda) & \text{ if $k<1$}.
 \end{aligned}
 \end{equation}
\end{lemma}

As Theorem \ref{theorem:DJM1} relates the polarization tensor for a domain with its exterior conformal mapping,  an estimate of the polarization tensor $M$ can be translated into an estimate of the coefficients of the Riemann mappings via this relation. In particular, we can obtain inequality relations of the coefficients of the interior conformal mapping by using Lemmas \ref{lemma:tensor_bounds} and \ref{lemma:monotonicity}. 

In view of the fact that Theorem \ref{theorem:DJM1} relates the PT to the exterior map coefficients, it may be more natural to state the inequality relations directly in terms of exterior conformal maps. However, in univalent function theory, it is more common to consider functions defined on the unit disk $\mathbb D$ so that, in the rest of this section, we derive inequality relations of the coefficients of the (interior) Riemann mappings.

Consider a Jordan domain (domain enclosed by a Jordan curve), namely $\widetilde{\Om}$, containing the origin. By the Riemann mapping theorem, there exists a unique (interior) conformal map $\Phi[\widetilde{\Om}]$ from the unit disk onto $\Omega$ with $\Phi(0)=0$ and $\Phi'(0)>0$, which admits the Taylor series expansion
\begin{equation} \label{eq:taylor}
 \Phi[\widetilde{\Om}](z) = b_1 z + b_2 z^2 + \cdots
\end{equation}
for some complex coefficients $b_n$, $n\geq2$, and $b_1>0$. The reflection of the exterior of $\widetilde{\Om}$ with respect to the unit circle centered at the origin defines a Jordan domain 
$$\Om=\left\{\frac{1}{z}:z\in\CC\setminus\overline{\widetilde{\Om}}\right\}\cup\{0\}.$$
The exterior conformal mapping $\Psi$ associated with $\Om$ satisfies $\gamma=\frac{1}{b_1}$ and is expanded as the following Laurent series
\begin{equation}
\Psi[\Om] (z) =\frac{1}{\Phi[\widetilde{\Om}](\frac{1}{b_1 z})}= z -\frac{b_2 } {b_1 ^2}+ \frac 1 {{b_1 ^3} } \Big(\frac {b_2 ^2} {b_1} - b_3\Big)  \frac 1 {z}  + \sum_{n \geq 2} \frac {a_n}{z^n}
\end{equation}
for some complex coefficients $a_n$, $n\geq 2$. 
From Theorem \ref{theorem:DJM1}, we have
\begin{equation} \label{eq:translation}
M(\Omega, \pm\frac{1}{2}) = 2\pi 
\begin{bmatrix}
\ds\pm \frac{1}{b_1 ^2} + \frac{1}{b_1 ^3}\, \text{Re} \Big(\frac{b_2 ^2 } {b_1} - b_3\Big) & \ds\frac{1} {b_1 ^3} \,\text{Im} \Big(\frac{b_2 ^2 } {b_1} - b_3\Big) \\
\ds \frac{1} {b_1 ^3} \,\text{Im} \Big(\frac{b_2 ^2 } {b_1} - b_3\Big)& \ds\pm \frac{1}{b_1 ^2} - \frac{1}{b_1 ^3} \, \text{Re}\Big(\frac{b_2 ^2 } {b_1} - b_3\Big)
\end{bmatrix}.
\end{equation}

Now, we derive the following theorem by translating Lemmas \ref{lemma:tensor_bounds} and \ref{lemma:monotonicity} into the coefficients of the interior Riemann mapping of a Jordan domain $\widetilde{\Om}$.
\begin{theorem} \label{thm:translation}
Suppose that $\widetilde{\Om}$ is a domain enclosed by a smooth Jordan curve, containing $0$. We denote its interior conformal mapping by 
$\Phi[\widetilde \Om] (z) = b_1 z + b_2 z^2 + \cdots.$
Then, it holds that
$$\frac{1}{b_1 ^2} \pm \frac{1}{b_1 ^3} \operatorname{Re} \Big(\frac{b_2 ^2 } {b_1} - b_3\Big) \geq \frac {| \Om|} { 2\pi}.$$
 In particular, if $|\Om|=1$, then we have $\frac{1}{b_1 ^2} \pm \frac{1}{b_1 ^3} \operatorname{Re} \big(\frac{b_2 ^2 } {b_1} - b_3\big) \geq \frac{1}{2\pi}$.

Further, suppose that $\widetilde V$ is another domain enclosed by a smooth Jordan curve satisfying $\widetilde \Om \subseteq \widetilde V$ (so that $ V \subseteq   \Om$). We denote its interior conformal mapping by $\Phi[\widetilde V] (z) = c_1 z + c_2 z^2 + \cdots$.
 Then, the following monotonicity inequality holds:
$$ \frac{1}{b_1 ^2} \pm \frac{1}{b_1 ^3}\, \operatorname{Re} \Big(\frac{b_2 ^2 } {b_1} - b_3\Big) \geq\frac{1}{c_1 ^2} \pm \frac{1}{c_1 ^3}\, \operatorname{Re} \Big(\frac{c_2 ^2 } {c_1} - c_3\Big).$$

\end{theorem}

\begin{proof}
{ From Lemmas \ref{lemma:tensor_bounds} and \ref{lemma:monotonicity} with $k=0,\infty$ and $H(x)=c_1 x_2+c_2 x_2$ for arbitrary constants $c_1,c_2$, we have
		\begin{equation} \label{eq:matrixinequalities}
	\begin{aligned}
	|\Omega|I_{2\times 2} &\leq M(\Om,+\frac{1}{2}),\\
	-| \Omega|I_{2 \times 2 } &\geq M(\Om,-\frac{1}{2}),\\
    M(V,+\frac{1}{2}) &\leq M(\Om,+\frac{1}{2}),\\
    M(V,-\frac{1}{2}) &\geq M(\Om,-\frac{1}{2}),
	\end{aligned}
	\end{equation}	
where we write $A \leq B$ for real symmetric matrices $A$ and $B$ such that $B-A$ is positive semi-definite.}
Combined with Theorem \ref{theorem:DJM1}, the first or the second inequality in \eqref{eq:matrixinequalities} proves the first assertion, and
the third or the fourth inequality in \eqref{eq:matrixinequalities} proves the next assertion.
\end{proof}

Recall that $\widetilde \Om = \Phi[\widetilde \Om](\mathbb D)$ and $\Phi[\widetilde \Om](z) = b_1 z +\mbox{higher order terms}$ ($b_1>0$). Therefore, it is quite reasonable to expect that large $b_1$ implies large $|\widetilde \Om|$ (so that small $|\Om|$). 
We may further establish more properties on the conformal mapping coefficients by exploiting our vast knowledge of GPT, e.g. we may apply Lemma \ref{lemma:monotonicity} with higher order $H(x)$. 

It is worth remarking that inequalities between coefficients of the Riemann mappings are of interest in univalent function theory \cite{Duren:1983:UF, Kiryatskii:1990:SFC, Lewin:1971:BFC, Pommerenke:1975:UF:book,Thomas:2018:UF:book}.

\begin{remark}
{
	Theorem \ref{thm:translation} can be extended to arbitrary Jordan domains (no smoothness assumptions on the boundary curves) $\widetilde \Om$ and $\widetilde V$ with $0 \in \widetilde \Om \subset \subset \widetilde V$. This is because given a Jordan domain $0 \in \widetilde U$, we can approximate $\widetilde U$ by domains enclosed by $\Phi[\widetilde U](\{z:|z|=r\})$ ($0<r<1$).}
\end{remark}

\section{Spectral monotonicity of the NP operator with respect to the shape deformation}\label{sec:NP:mono}
 
In this section, we numerically compute eigenvalues of the NP operator for various Fourier modes in the shape of the inclusion. 
We find that the eigenvalues have a monotonic behavior as the shape of the inclusion linearly deviates from a disk. 
 
 \subsection{Finite section method}

 For a $C^{1,\alpha}$ domain $\Om$, the NP operator on $H_0^{-1/2}(\p\Om)$ admits only a sequence of eigenvalues that accumulates to zero as its spectrum. In two dimensions, the spectrum is symmetric with respect to zero \cite{Blumenfeld:1914:UPF}. The positive eigenvalues are enumerated in the following way: 
 	\begin{equation}\label{eqn:eigenvalueorder}
		0.5>\left|\lambda_1\right|\geq\left|\lambda_2\right|\geq\left|\lambda_3\right|\geq\left|\lambda_4\right|\geq\cdots.
\end{equation}
The $l^2(\CC)$ space generated by the basis $\{\zeta_{\pm m}\}_{m\in\NN}$ corresponds the Sobolev space $H^{-1/2}_0(\p\Om)$ \cite{Jung:2018:SSM}; see also \cite{Kang:2018:SPS,Krein:1998:CLO} for the permanence of the spectrum for the NP operator with different norms.
For each $n\in\NN$, we set $H_n=\mbox{span}\big\{\zeta_{-n},\zeta_{-n+1},\cdots,\zeta_{-1},\zeta_1,\cdots,\zeta_{n-1},\zeta_n\big\}$ and define $P_n$ as the orthogonal projection to $H_n$. We then denote by $\left[\Kcal^*_{\p\Om}\right]_n$ the $n$-th section of $\left[\KstarOmega\right]$, that is the $2n\times2n$ matrix $$\left[{\KstarOmega}\right]_n=P_n\left[\KstarOmega\right]P_n=P_n \KstarOmega P_n.$$
For example, we have
$$ \left[\Kcal^*_{\p\Om}\right]_1= \begin{bmatrix}
0 & \mu_{1,1} \\
\overline{\mu_{1,1}} & 0 
\end{bmatrix}.
$$
For a bounded self-adjoint operator on a separable complex (infinite-dimensional) Hilbert space, the spectrum outside the convex hull of the essential spectrum is a set of eigenvalues with finite algebraic multiplicity, and can be approximated by eigenvalues of the finite section matrices \cite[Theorem 3.1]{Bottcher:2001:AAN}. Combining this result with the fact that $\Kcal^*_{\p\Om}$ is self-adjoint on $l^2(\CC)$ generated by $\{\zeta_{\pm m}\}$, we conclude that eigenvalues of the $n$-th section of $[{\KstarOmega}]_n$ converge to the spectrum of $\Kcal^*_{\p\Om}$; see also \cite[section 6.2]{Jung:2018:SSM}. 
Since $[{\KstarOmega}]_n$ is a self-adjoint finite-dimensional matrix, one can easily compute the eigenvalues.

In the following examples, we show positive eigenvalues $\lambda_k$, $1\leq k\leq 30$, of various smooth domains prescribed by their exterior conformal mapping, obtained by evaluating eigenvalues of $\left[\Kcal^*_{\p\Om}\right]_n$ with large $n$. The eigenvalue calculations were performed by MATLAB R2018a by use of the finite section method explained above. More precisely, we gradually increase $n$ from the initial value $100$ with increment step size $100$, namely $n= 100*step$, until the following is satisfied:\\
\smallskip
\textbf{Stopping condition.} 
For all $1\leq k\leq 30$ with five consecutive $step$, it holds that
\begin{equation}\label{eq:threshold}
	r_k^{(step)}= \frac{\left|\lambda_k^{(step)}-\lambda_k^{(step-1)}\right|}{\left|\lambda_k^{(step)}\right|}
	 <10^{-5},
\end{equation}
where $\lambda_k^{(step)}$ denotes the $k$-th eigenvalue of $\left[\Kcal^*_{\p\Om}\right]_n$ with $n=100*step$.
Even if the above mentioned condition does not hold, we terminate the calculation at $step = 16$ to avoid excessive computing time and accept the computed value $\lambda_k^{(16)}$ for $k$ satisfying the relative error condition \eqref{eq:threshold} for all $step =12,\dots,16$.

\smallskip

Figure \ref{FigErrors} shows the maximum relative error, $\max_{1\leq k\leq 30} r_k^{(step)}$, against $step=2,\dots,16$ for the smooth domain given by its exterior conformal mapping
\beq\label{Om:ex0}
\Psi(z) =z+ 0.01 z^{-1} + 0.07 z^{-2} + 0.01z^{-3} + 0.03 z^{-4} + 0.05 z^{-5} + 0.07 z^{-6}
\eeq
with $\gamma=1$.   
 The graph demonstrates that the relative error, $r_k^{(step)}$, decays at least exponentially, and the stopping condition is met when $step=11$ (i.e., $n=1100$). 
Numerical experiments show that such ``spiky" domains tend to exhibit slower convergence for $\lambda_k^{(step)}$. Other domains in the examples below show faster convergence compared to that in Figure \ref{FigErrors}. 
\begin{figure}[h!] \centering
\hskip -1.5cm
	\begin{subfigure}{0.3\linewidth}	
	
	\includegraphics[width=150pt, height=140pt]{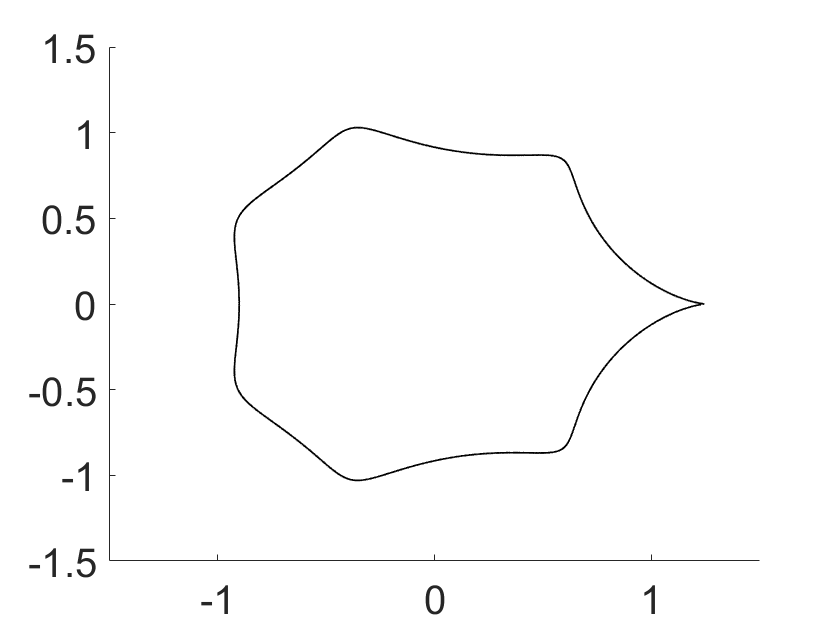}
	\end{subfigure}
	\hskip .3cm
	\begin{subfigure}{0.4\linewidth}	
	\includegraphics[width=250pt, height=140pt]{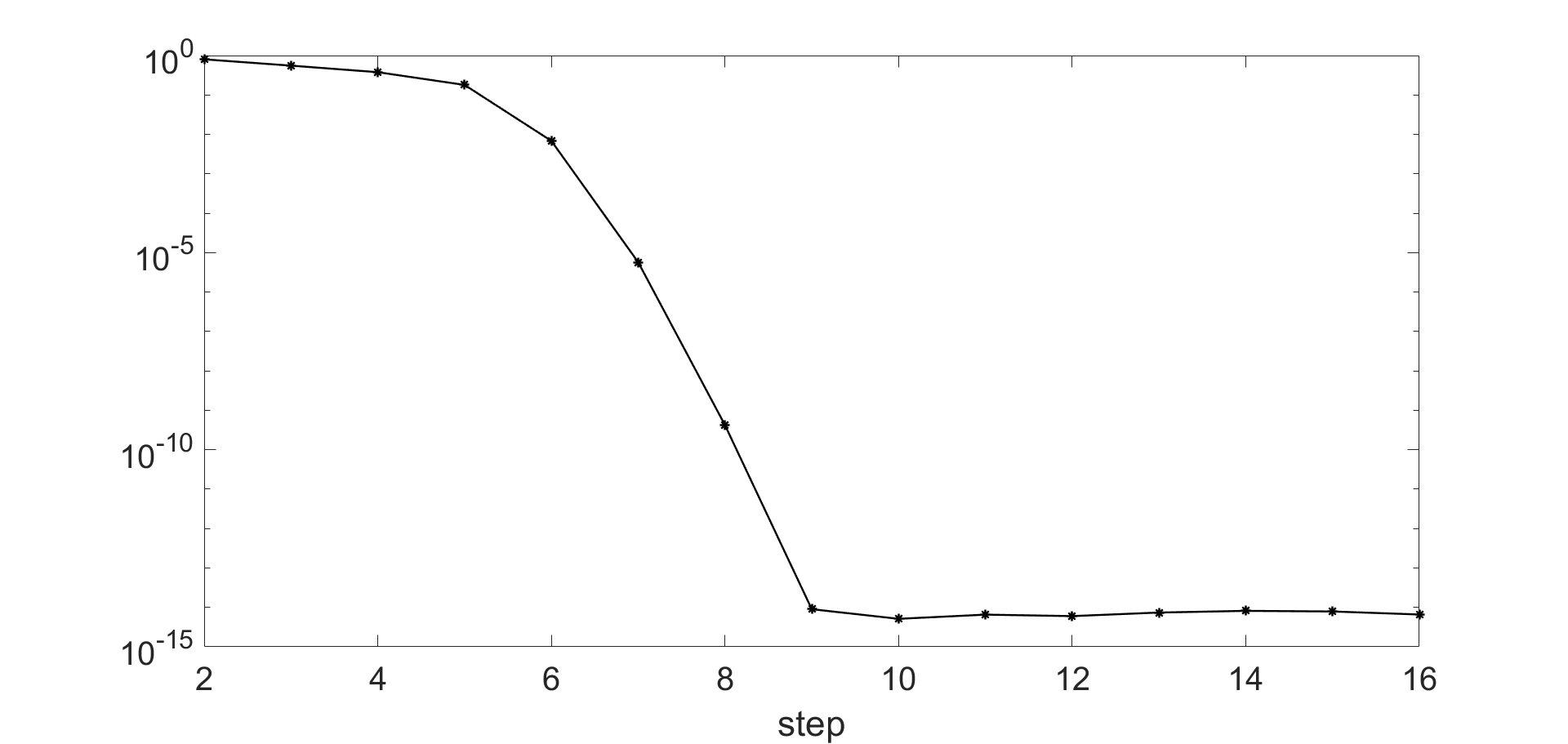}
	\end{subfigure}
\caption{The maximum relative error, $\max_{1\leq k\leq 30} r_k^{(step)}$, against $step=2,\dots,16$ in log scale (right figure) for $\Om$ given by \eqnref{Om:ex0} (left figure). 
}\label{FigErrors}
\end{figure}

Note that for a smooth domain, $\lambda_k$ exponentially decays as $k$ increases (\cite{Ando:2018:EDE,Jung:2020:DEE}), e.g., $\lambda_k=\pm \frac 1 2 \left(\frac{a-b}{a+b}\right)^k$ for an ellipse with major axis $a$ and minor axis $b$ are \cite{Ammari:2007:BIM}. 
Some eigenvalues with large $k$ tend to be small enough that they do not meet the stopping condition due to round-off errors.
It turns out that for such cases in the following examples, the last approximated value $\lambda^{(16)}_k$ satisfies 
$
 0<\lambda^{(16)}_k<2.6161 \times 10^{-13}.
$

\subsection{Examples}\label{subsec:examples}

As the following examples will illustrate, the monotony appearing in the above example again makes an appearance whenever the coefficients of $\Psi$ vary linearly in a parameter.
However, the spectral monotonicity does not hold in the general case where the coefficients vary non-linearly (see Remark \ref{remark:nonlin}).

\begin{example}\rm
 We take $\Omega=\Om_{m,s}$ to be the domain given by its exterior conformal mapping 
\beq\label{Psi:fourier}
\Psi(z) = z + \frac{s}{m}{z^{-m}},\quad m=1,\dots,6,\ s = 0.01,0.02, \dots, 0.99,
\eeq
with $\gamma=1$; $s$ is set to be smaller than $1$ for $\Psi$ with $|z|=1$ defining a Jordan curve. 
The Neumann--Poincar\'{e} eigenvalues  $\lambda_k$ are computed by the algorithm described above.

In Figure \ref{Fig:1to6}, in each sub-figure corresponding to one of the values of $m$, we plot the nine graphs of $\lambda_k$ associated with the domain $\Om_{m,s}$ with $s=0.1, 0.2, \cdots, 0.9$. 
Note that, without any exception, the graphs corresponding to larger $s$ are strictly above those corresponding to smaller $s$. That is, for each fixed $m$ and $k$, the $k$-th eigenvalue corresponding to $\Omega_{m,s}$ is strictly increasing in $s$.

Additionally, eigenvalues of the perturbation of a disk with a higher Fourier mode, $m$, tend to cluster; if $m$ is larger, then the ``clusters" get bigger (their size being $m+1$).
 This behavior accords with the analytic asymptotic results obtained in \cite{Ammari:2019:SRN}, where the eigenvalues of the NP operator of the algebraic domain with $\Psi(z)=z+\frac{\delta}{ z^{m}}$ ($\delta>0$ small) and $\gamma = 1$ are shown to satisfy
 \beq
 \begin{cases}
\ds\frac{\delta}{2} \times \left\{\, \pm\sqrt {1 \cdot m}, \, \pm\sqrt{2 \cdot (m-1)}, \,\cdots,\, \pm\sqrt {k \cdot k} \,\right\} + O(\delta ^2)   \quad&\mbox{if }m= 2k-1,\\[2mm]
\ds\frac{\delta}{2} \times \left\{\, \pm\sqrt {1 \cdot m}, \, \pm\sqrt{2 \cdot (m-1)},\, \cdots, \,\pm\sqrt {k \cdot (k+1)}\, \right\} + O(\delta ^2) 
\quad&\mbox{if }m= 2k.
\end{cases}
\eeq

\begin{figure}[p]
	\centering
	\begin{subfigure}[b]{0.49\linewidth}
		\centering\includegraphics[width=240pt, height=160pt]{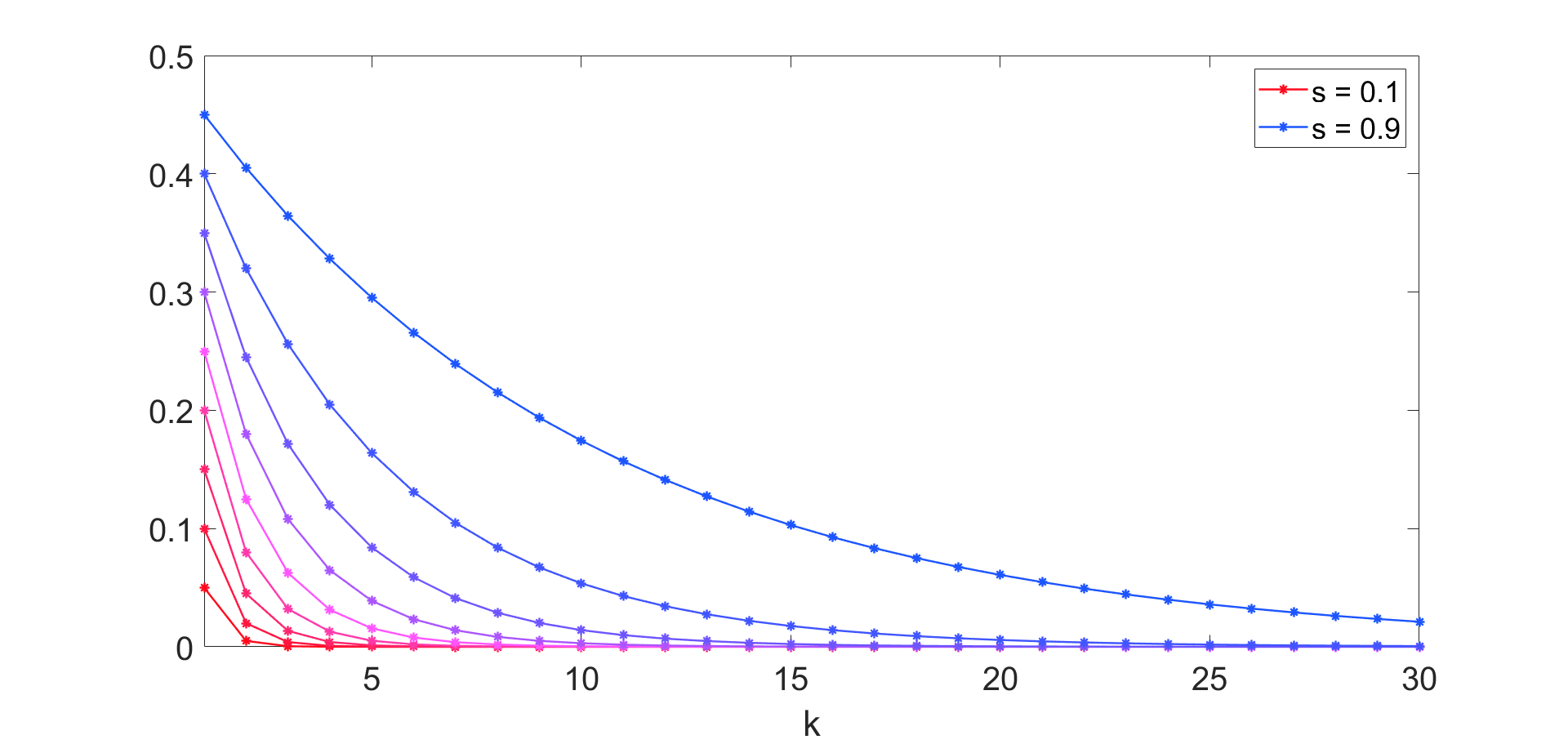}
		\caption{$m=1$\label{Fig1}}
	\end{subfigure}%
	\begin{subfigure}[b]{0.49\linewidth}
		\centering\includegraphics[width=240pt, height=160pt]{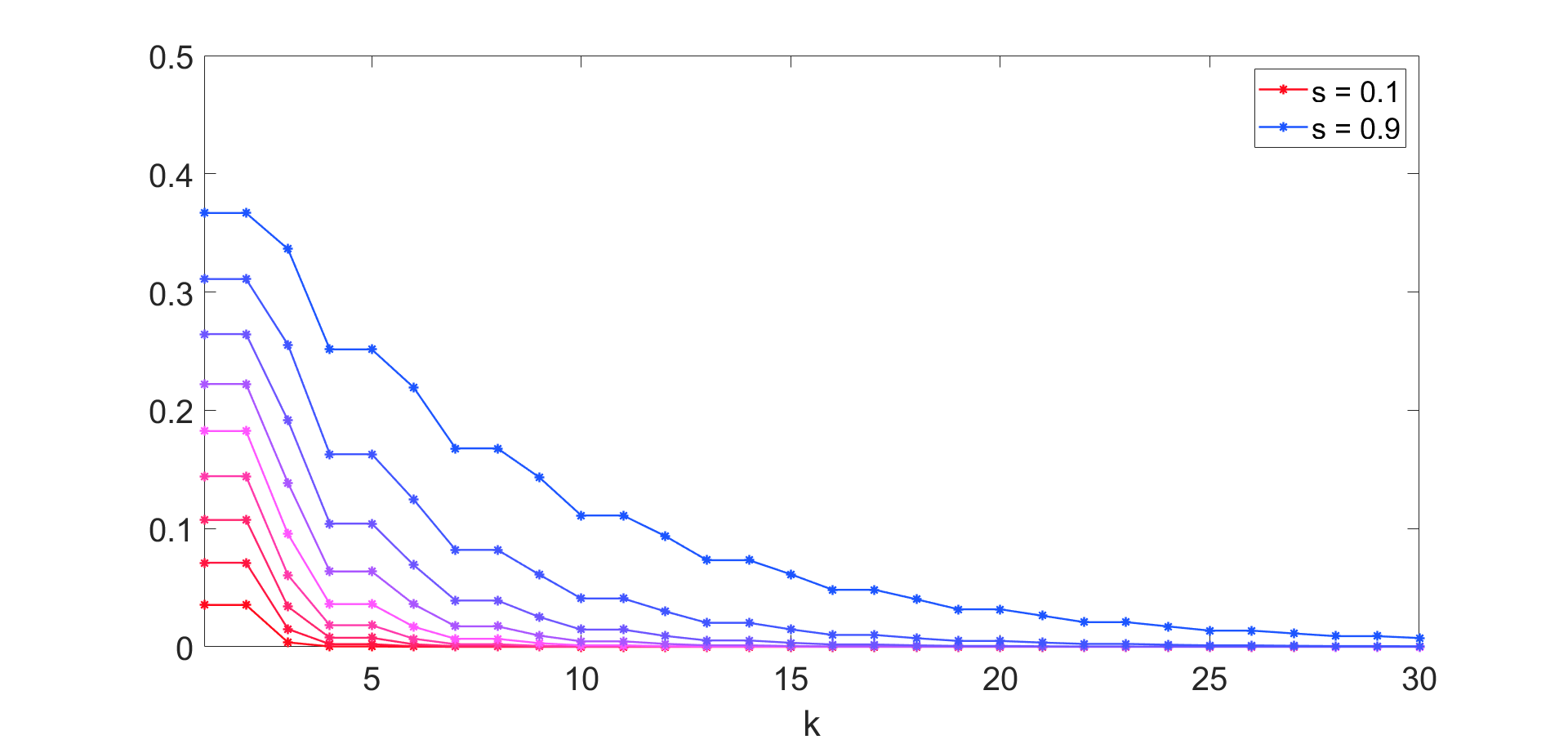}
		\caption{$m=2$\label{Fig2}}
	\end{subfigure}

\begin{subfigure}[b]{0.49\linewidth}
	\centering\includegraphics[width=240pt, height=160pt]{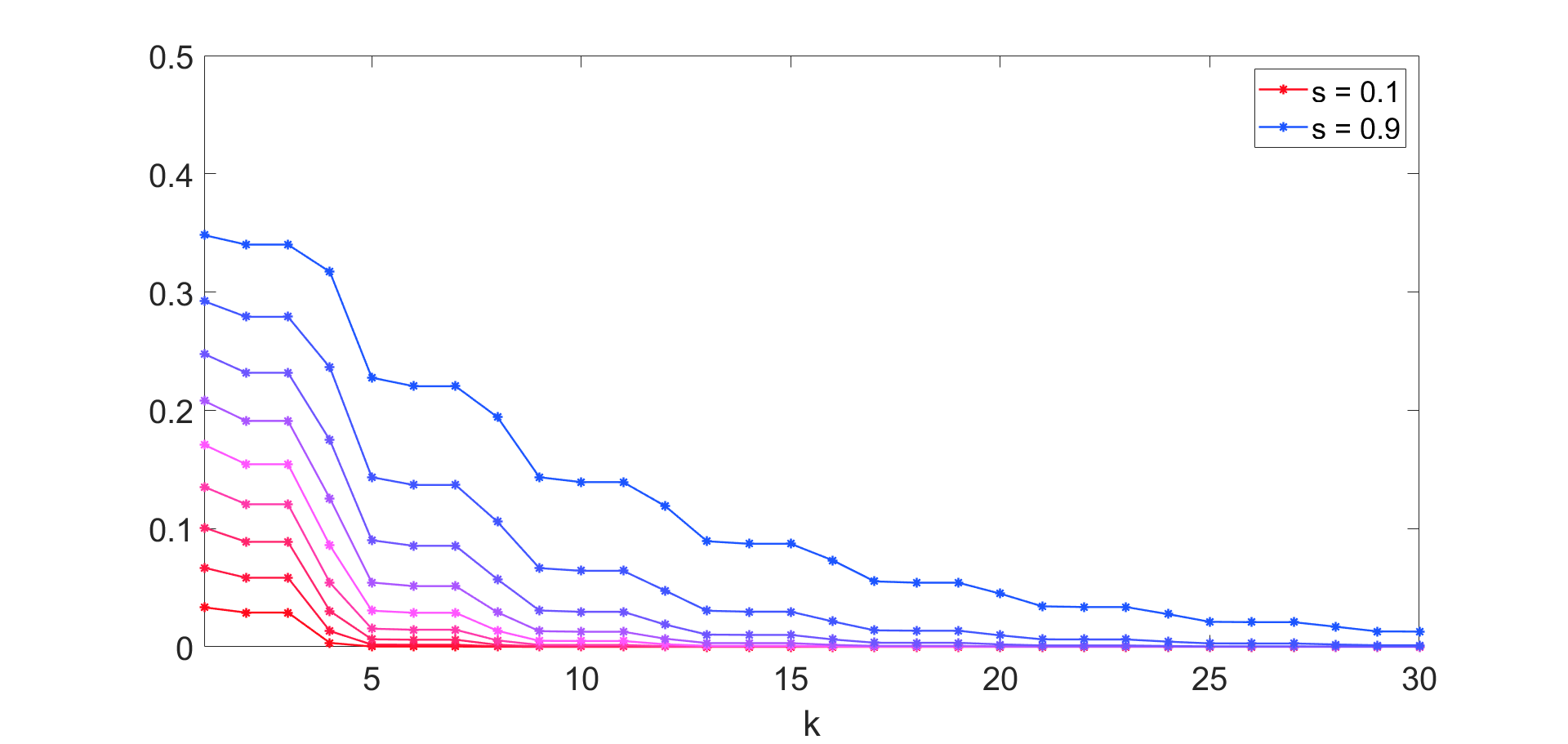}
	\caption{$m=3$\label{Fig3}}
\end{subfigure}%
\begin{subfigure}[b]{0.49\linewidth}
	\centering\includegraphics[width=240pt, height=160pt]{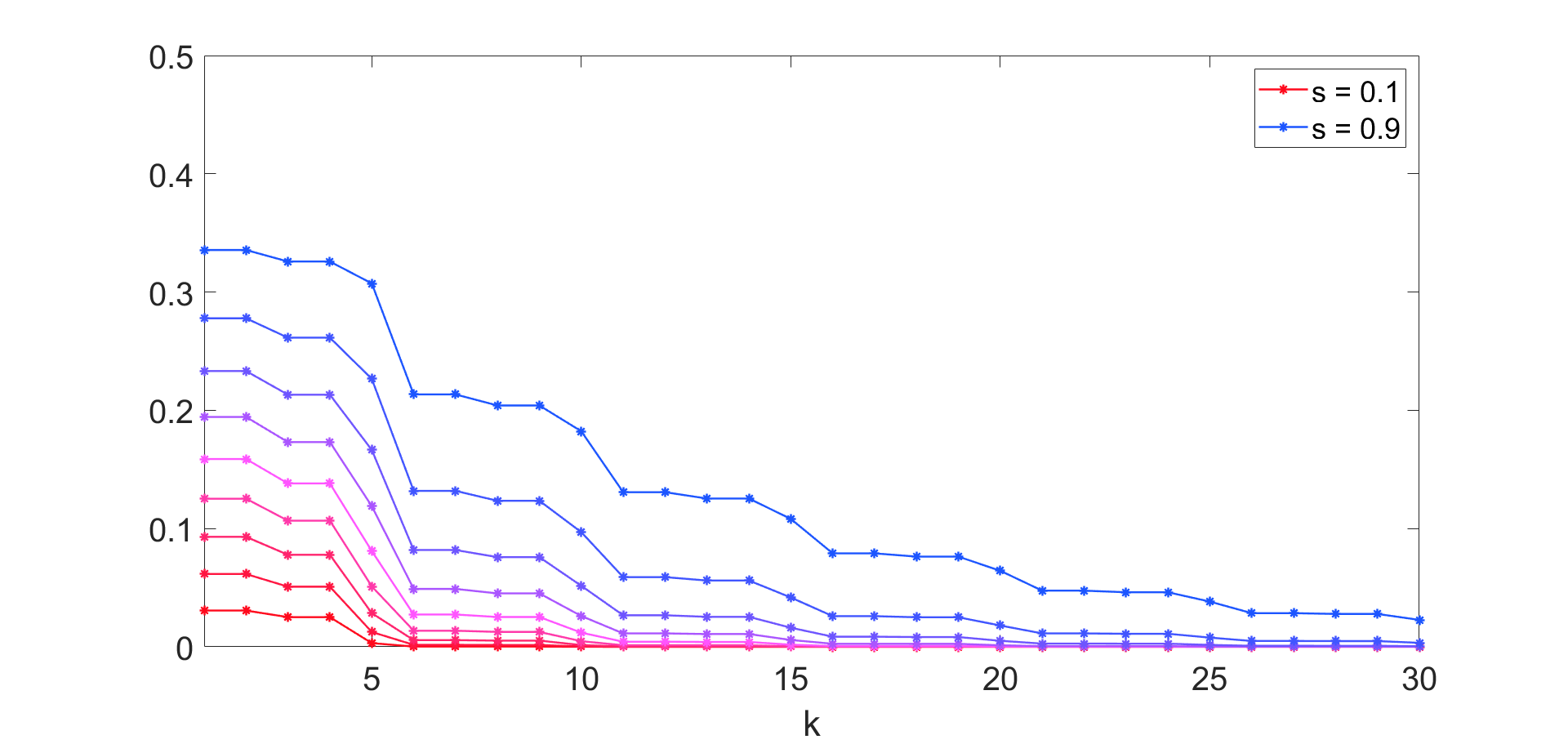}
	\caption{$m=4$\label{Fig4}}
\end{subfigure}

\begin{subfigure}[b]{0.49\linewidth}
	\centering\includegraphics[width=240pt, height=160pt]{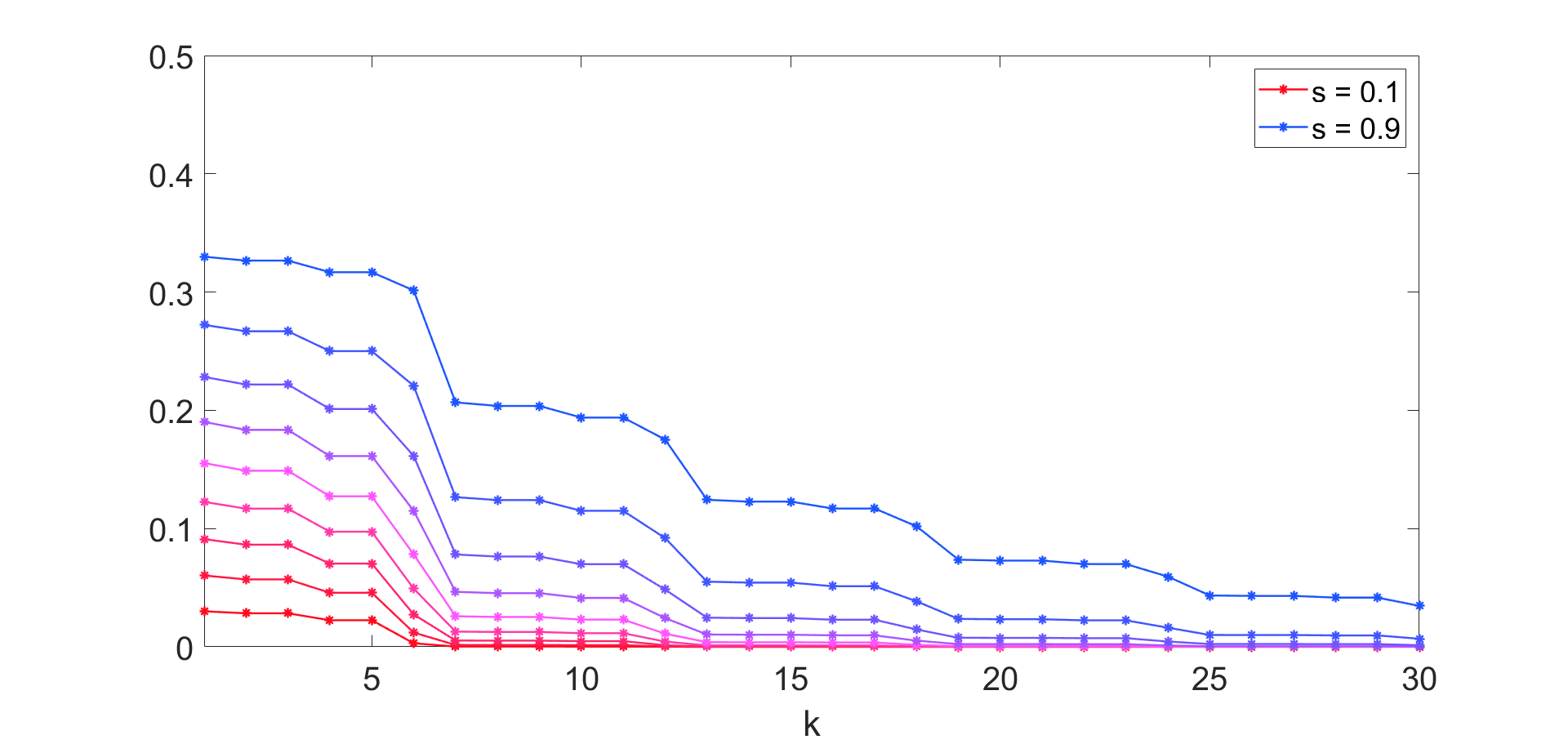}
	\caption{$m=5$\label{Fig5}}
\end{subfigure}
\begin{subfigure}[b]{0.49\linewidth}
	\centering\includegraphics[width=240pt, height=160pt]{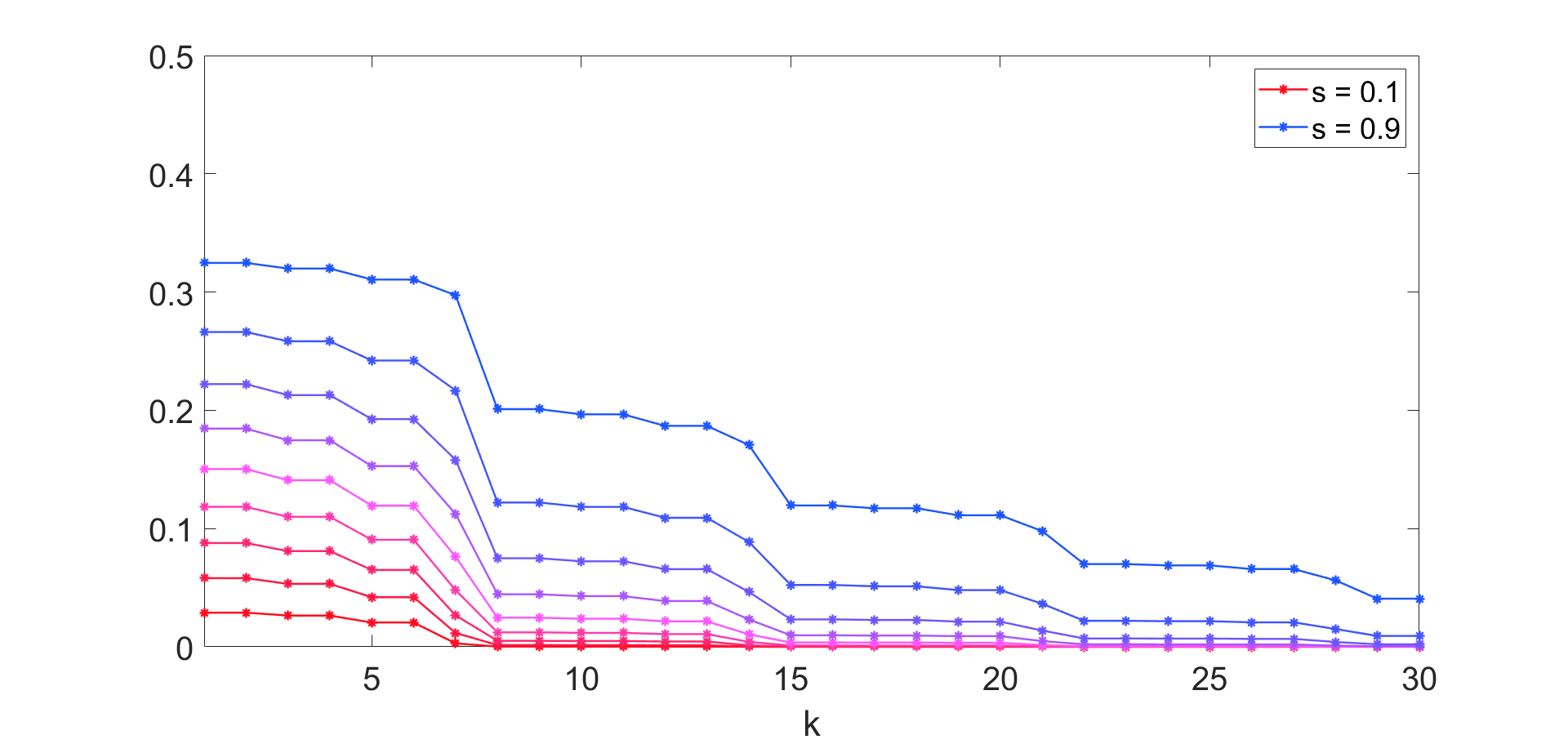}
	\caption{$m=6$\label{Fig6}}
\end{subfigure}

	\caption{ In each sub-figure corresponding to one of values of $m$, we plot the nine graphs of $\lambda_k$ associated with the domain $\Om_{m,s}$ given by \eqnref{Psi:fourier} with $s=0.1, 0.2, \cdots, 0.9$.  For each $k$, the eigenvalues have a monotonic behavior as the shape of the inclusion deviates from a disk (i.e., as $s$ increases). }
	\label{Fig:1to6}
\end{figure}
\end{example}

\begin{example}\rm
In this example, we take $\Om=\Omega_j$ to be the domain given by its exterior conformal map 
\beq\label{Psi:Ex2}
\Psi(z) = z + j\left(\frac 1 {600} z^{-1} + \frac 1 {300} z^{-2} + \frac 1 {1200} z^{-3} + \frac 1 {320} z^{-4}\right),\quad 1 \leq j \leq 42.
\eeq
For all $1 \leq j \leq 42$ and $1\leq k \leq 30$, the eigenvalues $\lambda_k$ were successfully computed by following the stopping criterion \eqnref{eq:threshold}. 
The obtained value of the $k$-th eigenvalue is strictly increasing in $j$ for each $k$ without any exception. 
Figure \ref{Fig13_14}\,(a) shows the nine graphs of $\lambda_k$ associated with the domain $\Om_j$ with $j=1,6,\dots,41$.
\end{example}
	
\begin{example}\rm
	We now take $\Om=\Omega_j$ to be the domain given by its exterior conformal map 
	
	\beq\label{Psi:Ex3}
	\Psi_j(z) = z + j\left( \frac 1 {400} z^{-1} + \frac 1 {600} z^{-2} + \frac 1 {1200} z^{-3} \right),\quad 1 \leq j \leq 119.
	\eeq
For $1 \leq j \leq 119$ and $1\leq k \leq 30$, the eigenvalues $\lambda_k$ were successfully computed by following the stopping criterion \eqnref{eq:threshold} except some cases, in which $\lambda^{(16)}_k$ satisfies $
 0<\lambda^{(16)}_k<2.6161 \times 10^{-13}.
$
The obtained value of $\lambda_k$ is strictly increasing in $j$, for all $k$ without any exception. 
Figure \ref{Fig13_14}\,(b) shows the graphs of $\lambda_k$ against $k$ for $\Om=\Om_j$ with $j=1,11,\dots,111$.
	 \begin{figure}[h!]
		\centering
		\begin{subfigure}{0.99\textwidth}
		\centering
		\includegraphics[scale=0.25]{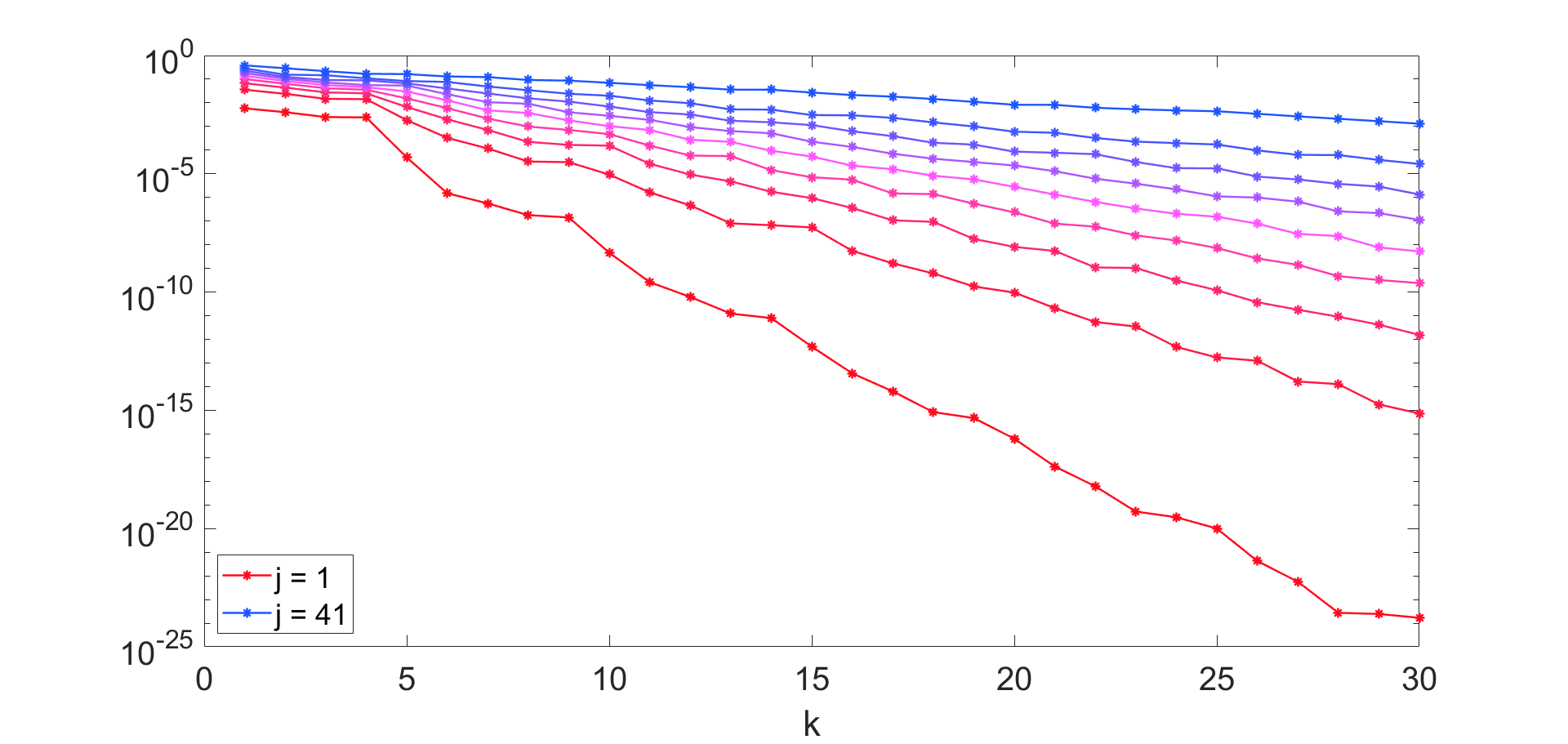}
		\caption{Eigenvalues corresponding to the domain given by \eqnref{Psi:Ex2} in log scale}
		\label{Fig13}
	\end{subfigure}
	
	\begin{subfigure}{0.99\textwidth}
	\centering
		\includegraphics[scale=0.25]{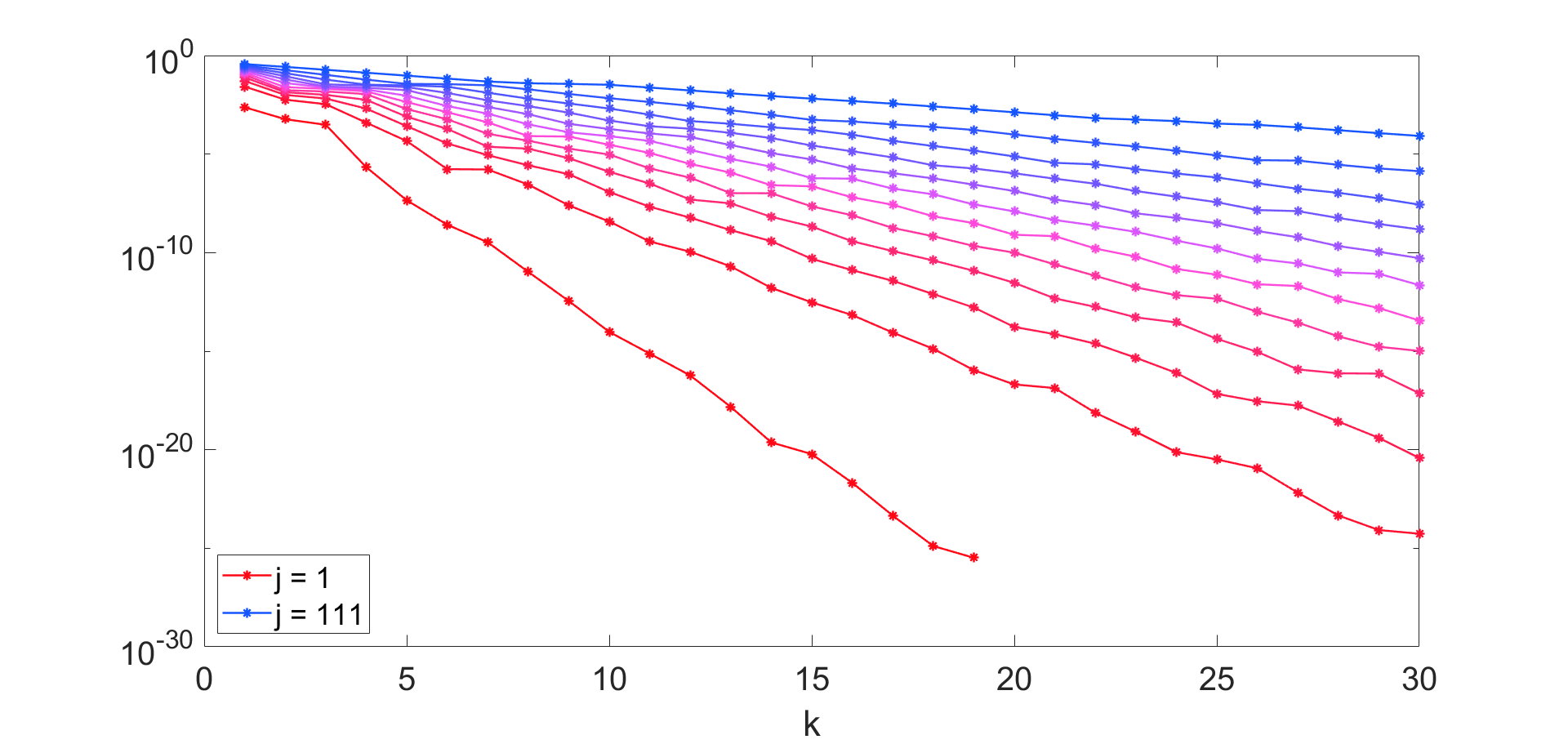}
		\caption{Eigenvalues corresponding to the domain given by \eqnref{Psi:Ex3} in log scale}
		\label{Fig14}
		\end{subfigure}
		\caption{ For each $k$, the eigenvalues have a monotonic behavior as the shape of the inclusion deviates from a disk (i.e., as $j$ increases).}\label{Fig13_14}
	\end{figure}
	\end{example}

\begin{remark}\label{remark:nonlin}
Consider $\Omega_j$, the domain given by its exterior conformal map $\Psi_j(z) = z+ \frac {j^2} {600} z^{-1} + \frac j {300} z^{-2}$. It turns out by the numerical computation that $\lambda_2=4.7155 \cdot 10^{-3}$ for $j=5$ and $\lambda_2=4.4402 \cdot 10^{-3}$ for $j=6$. In other words,
$\lambda_2(\Om_5) > \lambda_2(\Om_6),$ which demonstrates that the spectral monotonicity does not hold in the general case where the coefficients non-linearly vary.		
\end{remark}

\section{Conclusion}\label{ref:conclusion}
We investigated the geometric dependence of properties of composite materials based on the series expansions of the layer potential operators obtained in \cite{Jung:2018:SSM}. 
We obtained an explicit relation between the trace of the polarization tensor and the exterior conformal mapping corresponding to a planar inclusion with extreme or near-extreme conductivity.  
We then derived an explicit asymptotic for the effective conductivity of cylindrical periodic conductivity composites. 
Also, we observed a monotonic behavior of the spectrum of the NP operator when the shape of the inclusion linearly varies from a disk, by numerically computing eigenvalues of the NP operator for various example shapes of inclusions. 
Additionally, we derived inequality relations between the coefficients of the Riemann mapping associated with a general simply connected bounded Lipschitz domain by employing the properties of the PT corresponding the domain with extreme conductivity. It will be an interesting challenge to extend these results to inclusions with arbitrary conductivities.

\ifx \bblindex \undefined \def \bblindex #1{} \fi\ifx \bbljournal \undefined
  \def \bbljournal #1{{\em #1}\index{#1@{\em #1}}} \fi\ifx \bblnumber
  \undefined \def \bblnumber #1{{\bf #1}} \fi\ifx \bblvolume \undefined \def
  \bblvolume #1{{\bf #1}} \fi\ifx \noopsort \undefined \def \noopsort #1{} \fi

\end{document}